\documentclass[12pt]{amsart}
\usepackage{amssymb,amscd,amsthm,verbatim,nicefrac,amsmath,color,fancyhdr,mathrsfs,braket,graphicx, turnstile,hyperref,indentfirst,csquotes,url,enumitem,amsfonts, mathtools,lipsum}
\usepackage[letterpaper, left=2.5cm, right=2.5cm, top=2.5cm,
bottom=2.5cm,dvips]{geometry}
\hypersetup{
    colorlinks=true,
    linkcolor=blue,
    filecolor=magenta,      
    urlcolor=cyan,
    citecolor=blue,
    pdftitle={$2$nd-Order $L^\infty$ Variational Problems with Lower-Order Terms},
    pdfpagemode=FullScreen,
    }

\numberwithin{equation}{subsection}
\let\oldsection\section% Store \section
\renewcommand{\section}{% Update \section
  \renewcommand{\theequation}{\thesection.\arabic{equation}}% Update equation number
  \oldsection}% Regular \section
\let\oldsubsection\subsection% Store \subsection
\renewcommand{\subsection}{% Update \subsection
  \renewcommand{\theequation}{\thesubsection.\arabic{equation}}% Update equation number
  \oldsubsection}% Regular \subsection
\title{On Second-Order $L^\infty$ Variational Problems with Lower-Order Terms}

\author{Ben Dutton}

\address{B.D., Department of Mathematical Sciences,
University of Bath, Bath BA2 7AY, UNITED KINGDOM}

\email{bgd27@bath.ac.uk}

\author{Nikos Katzourakis}

\address{N.K. (corresponding author), Department of Mathematics and Statistics, University of Reading, Whiteknights Campus, Pepper Lane, Reading RG6 6AX, UNITED KINGDOM}

\email{n.katzourakis@reading.ac.uk}

\makeatletter
\@namedef{subjclassname@2020}{\textup{2020} Mathematics Subject Classification}
\makeatother



\thanks{\!\!\!\!\!\!\!\texttt{B.D. has been financially supported by an Undergraduate Research Bursary URB-2023-95 from the London Mathematical Society and a scholarship from the EPSRC CDT in Statistical Applied Mathematics at Bath (SAMBa), under the project EP/Y034716/1. \\ N.K. has been partially financially supported through the EPSRC grant EP/X017206/1.}}

\def\R{\mathbb{R}}
\def\B{\mathbb{B}}

\def\H{\mathrm{H}}
\def\E{\mathrm{E}}
\def\J{\mathrm{J}}
\def\A{\mathrm{A}}
\def\D{\mathrm{D}}

\def\X{\mathrm{X}}

\def\bZ{\mathrm{\mathbf{Z}}}
\def\p{\mathrm{p}}

\def\d{\mathrm{d}}

\renewcommand\O{\mathcal{O}}
\renewcommand\phi{\varphi}

\newcommand{\av}{-\hspace{-13pt}\displaystyle\int}

\def\weaklyconv{-\!\!\!\!\!-\!\!\!\!\rightharpoonup}
\def\longconv{-\!\!\!\!-\!\!\!\rightarrow}
\newcommand{\overbar}[1]{\mkern1.5mu\overline{\mkern-1.5mu#1\mkern-1.5mu}\mkern 1.5mu}
\def\weakstar{\overset{*_{\phantom{|}}}{{\smash{\weaklyconv }}\,}}

\newtheorem{thm}{Theorem}[section]

\newtheorem{lem}[thm]{Lemma}
\newtheorem{rmk}[thm]{Remark}
\newtheorem{cor}[thm]{Corollary}

\newtheorem{defn}[thm]{Definition}

\newtheorem{der}[thm]{Derivation}

\begin{document}

\maketitle

\begin{abstract}
In this paper we study $2$nd order $L^\infty$ variational problems, through seeking to minimise a supremal functional involving the Hessian of admissible functions as well as their lower-order terms, considering for fixed $\Omega\subseteq\mathbb R^n$ open, and $\H : \Omega\times\big(\mathbb R \times\mathbb R^n \times \mathbb R^{n^{\otimes2}}_s \big) \to \mathbb R$, the functional
\[
\ \ \ \ \ \E_\infty(u,\mathcal{O}) :=\underset{\mathcal{O}}{\mathrm{ess}\sup}\hspace{1mm}\mathrm H (\cdot,u,\mathrm D u,\mathrm D^2u )  , \ \ u\in W^{2,\infty}(\Omega), \  \mathcal{O} \subseteq \Omega \text{ measurable}.
\] 
Specifically, we establish the existence of minimisers subject to (first-order) Dirichlet data on $\partial \Omega$ under natural assumptions, and, when $n=1$, we also show the existence of absolute minimisers. We further derive a necessary fully nonlinear PDE of third-order which arises as the analogue of the Euler-Lagrange equation for absolute minimisers, and is given by 
  $$
\ \  \mathrm H_{\mathrm X}(\cdot,u,\mathrm D u,\mathrm D^2u): \mathrm D\big(\mathrm H(\cdot,u,\mathrm D u,\mathrm D^2u)\big)\otimes \mathrm D\big(\mathrm H(\cdot,u,\mathrm D u,\mathrm D^2u)\big)=0\ \  \text{ in }\Omega.
  $$ 
We then rigorously derive this PDE from smooth absolute minimisers, and prove the existence of generalised $\mathcal{D}$-solutions to the (first-order) Dirichlet problem on bounded domains. This generalises the key results obtained in \cite{Katzourakis-Pryer-2020} which first studied problems of this type, providing at the same time some simpler streamlined proofs. 
\end{abstract}

%%%%%%%%%%%%%%%%%%%%%%%%%%%%%%%%%%%%%%%%%%%%%%%%%%%%%%%%%%%%%%

\smallskip

\section{Introduction}

Let $\Omega\subseteq\R^n$ be open, $n\in\mathbb{N}$, and $\H : \Omega\times \big( \R\times\R^n\times\R^{n^{\otimes2}}_s \big)\to \R$ a Carath\'eodory function. The central object of study in this paper is the supremal (or $L^\infty$) functional 
\begin{equation}
\label{eqn:fullfunctional}
  \begin{aligned}
\ \ \ \ \E_\infty(u,\O)&:=\, \underset{x\in\O}{\mathrm{ess}\sup}\hspace{1mm} \H\big(x,u(x),\D u(x),\D^2u(x)\big),  \ \ \ \  u\in W^{2,\infty}(\Omega), \ \O\subseteq\Omega,
  \end{aligned}
\end{equation}
where $\O\subseteq\Omega$ is measurable, viewed as an extension of the functional 
\begin{equation}\label{eqn:onlyhessiandependence}
  (u,\O) \, \mapsto \, \underset{x\in\O}{\mathrm{ess}\sup}\hspace{1mm} \H(\D^2u(x))
\end{equation}
to include lower-order terms. The functional \eqref{eqn:onlyhessiandependence} was investigated in \cite{Katzourakis-Pryer-2020}, wherein the second appearing author and Pryer initiated the study of second-order $L^\infty$ variational problems. Herein we show that, if $\H\in C^1(\Omega\times\R\times\R^n\times\R^{n^{\otimes2}}_s)$, the necessary PDE which arises as the analogue of the Euler-Lagrange equation for the supremal functional \eqref{eqn:onlyhessiandependence} is the fully nonlinear PDE of third-order given by
\begin{equation}
\label{eqn:A2inftensor}
  \A^2_\infty u \, :=\, \H_\X(\J^2u):\D\big(\H(\J^2u)\big)\otimes\D\big(\H(\J^2u)\big)= \, 0 \ \ \text{ in }\Omega.
\end{equation}
In index form, \eqref{eqn:A2inftensor} can be rewritten as
\begin{equation}\label{eqn:A2infindex}
  \sum_{i,j=1}^n\H_{\X_{ij}}(\J^2u)\D_i\big(\H(\J^2u)\big)\D_j\big(\H(\J^2u)\big)=0 \ \ \text{ in }\Omega.
\end{equation} 
In the above, $\J^2u$ is the second-order jet of $u$, where in general the $k$-th order jet is the map
\begin{equation}
  \J^k u:=\left(\cdot,u,\D u,\D^2u,\ldots,\D^k u\right),\ \ k\in\mathbb{N},
\end{equation}
with the derivatives (of first, second, and $k$-th order) of $u \in C^k(\Omega)$ denoted by
\begin{equation}
  \begin{aligned}
    \D u=\left(\D_i u\right)_{i=1}^n&\ :\ \Omega\to\R^n,\\
    \D^2u=\left(\D^2_{ij}u\right)_{i,j=1}^n&\ :\ \Omega\to\R^{n^{\otimes2}}_s,\\
    \D^k u=\left(\D^k_{i_1...i_k} u\right)_{i_1,\ldots,i_k=1}^n&\ : \ \Omega\to\R^{n^{\otimes k}}_s,
  \end{aligned}
\end{equation}
and valued into their respective (symmetric) tensor space defined as 
\begin{equation}
  \R^{n^{\otimes k}}_s:=\bigg\{V\in\underbrace{\R^n\otimes\cdots\otimes\R^n}_{k \text{ times}}\ :\ V_{i_1\cdots i_k}=V_{\sigma(i_1\cdots i_k)}, \sigma\text{ permutation on }\{i_1,\ldots,i_k\}\bigg\}.
\end{equation} 
Finally, $A \!:\! B := \mathrm{tr} (A^\top B)$ denotes the Euclidean (Frobenius) inner product in $\R^{n^{\otimes 2}}$, and we symbolise the arguments of the supremand $\H$ throughout as $\H(x,\eta,\p,\X)$, and in particular $\H_\X,\H_\p$, and $\H_\eta$ denote the derivatives of $\H$ with respect to the subscripted arguments. Our general functional space and PDE notation is generally standard, as e.g.\ in \cite{Fonseca-Leoni-2007}, or self-explanatory.

The aim of this paper is to study variational problems for the general second-order supremal functional \eqref{eqn:fullfunctional}, generalised solutions to the corresponding third-order fully nonlinear PDE \eqref{eqn:A2inftensor}, as well as the deeper connection between these two objects.

Supremal variational problems were first considered in the 1960s in the work of Aronsson, the pioneer of this field (see \cite{Aronsson-1965, Aronsson-1966, Aronsson-1967, Aronsson-1984}, and \cite{Aronsson-1986}). One of the advantages of minimising the $L^\infty$-norm is that it provides a uniformly small pointwise energy, whereas minimising an integral norm may allow for large spikes of the maximum pointwise energy, though the area under the graph might be small. This former approach offers better models in applications when this difference is relevant, but the theory is interesting from a pure mathematical standpoint nonetheless. The (scalar) first-order theory of $L^\infty$ variational problems, in which the supremand depends on the gradient of admissible real-valued functions and possibly lower-order terms, is well established (see e.g.\,\cite{BJW-2001-Euler} and \cite{BJW-2001-Lsc} from the early 2000s). Without any attempt to be exhaustive, we refer to the following interesting relevant works: \cite{AP, KZ, MWZ, PP, PZ, RZ, Ribeiro-Zappale-2024}. Most notably, we observe in \cite{Bhattacharya-DiBenedetto-Manfredi-1989} a correspondence between absolute minimisers of first-order supremal functionals and viscosity solutions of the associated necessary PDE. 

The vectorial first-order case with $u:\R^n\supseteq\Omega\to\R^N$, $N\in\mathbb{N}$, proved substantially harder and is still under development (see e.g.\,\cite{Katzourakis-Abugirda-2016, Katzourakis-Ayanbayev-2017, Katzourakis-2015-AbsMin, Katzourakis-Shaw-2018} from the 2010s). All of this leads to the consideration of the (scalar) second-order case as a natural next step. An interesting conclusion of \cite{Katzourakis-Pryer-2020} is that the second-order case does not follow by analogy from the first-order case. Recent literature on the second-order case includes \cite{Clark-Katzourakis-2024, Katzourakis-Moser-2023, Katzourakis-Moser-2024-1Currents, Katzourakis-Parini-2017}, and \cite{Katzourakis-Moser-2024-LocalMinimisers}.

The lack of G\^ateaux differentiability of $\E_\infty$ (which is true regardless of the smoothness or convexity of the supremand $\H$) is a fundamental difficulty of the theory. Danskin's theorem on differentiating maxima \cite{Danskin-1967} could apply, but it requires too much regularity and that the $\mathrm{argmax}$ set be a singleton, which is unrealistic. This challenge is usually overcome using the following method. A central technique in exploring $L^\infty$ variational problems is to approximate the functional $\E_\infty$ using $L^p$ functionals (with $p\ge1$) of the form
\begin{equation}
\label{1.8}
  (u,\O) \mapsto \bigg(\av_{\O} \big| \H(x,u(x),\D u(x),\D^2u(x)) \big|^p\hspace{1mm}\d x\bigg)^{\frac1p},
\end{equation} 
and then pass to the limit as $p\to\infty$ to explore $L^\infty$ phenomena (temporarily ignoring the fact that $\H$ in \eqref{eqn:fullfunctional} might attain negative values). This is based on the observation that the $L^p$-norm tends to the $L^\infty$-norm as $p\to\infty$, provided the integrand is in the space $L^\infty(\Omega)$.

Let us define below the central notions of minimality which we will utilise in this paper.

\begin{defn}
[Global and absolute minimisers]
\label{defn:globalandabsoluteminimisers}
A function $u\in W^{2,\infty}_g(\Omega):=g+W^{2,\infty}_0(\Omega)$ is a global minimiser of \eqref{eqn:fullfunctional} on $\Omega$, if the following inequality holds:
\begin{equation}
  \E_\infty(u,\Omega)\le\E_\infty(u+\phi,\Omega), \ \forall\hspace{1mm}\phi\in W^{2,\infty}_0(\Omega). 
\end{equation}
A function $u\in W^{2,\infty}(\Omega)$ is an absolute minimiser of \eqref{eqn:fullfunctional} on $\Omega$ if the following inequality holds:
\begin{equation}
  \E_\infty(u,\Omega')\le\E_\infty(u+\phi,\Omega'),\ \forall\hspace{1mm}\Omega'\Subset\Omega,\ \forall\hspace{1mm}\phi\in W^{2,\infty}_0(\Omega').
\end{equation}
\end{defn}

Another major hindrance in $L^\infty$ is that minimisers are not automatically optimal on subdomains, in contrast to the integral calculus of variations. As such, minimisers need to be assumed to minimise on subdomains from the outset, which is the justification of the notion of absolute minimisers. Moreover, the PDEs which arise in the second-order case are fully nonlinear, of third-order, and not elliptic, necessitating the development of new forms of generalised solutions. This is a non-trivial task, as standard approaches fail, and in particular neither viscosity nor weak solutions apply. As shown in \cite{Katzourakis-Pryer-2020}, in general, solutions to \eqref{eqn:A2inftensor} cannot be in $C^3(\Omega)$, even when $n=1$, there are no lower-order terms, and we restrict our attention to (absolutely) minimising solutions of $u \mapsto \|u''\|_{L^\infty(\Omega)}$ (see section \ref{section:existenceofDsolutions} for more details). 

We therefore need to resort to appropriately defined generalised solutions to \eqref{eqn:A2inftensor}. To this aim, $\mathcal{D}$-solutions for the Dirichlet problem for $\A^2_\infty u=0$ are considered (see also \cite{Katzourakis-2015-AbsMin, Katzourakis-2017-Dsols,Croce-Katzourakis-Pisante-2017}). This approach was first introduced in \cite{Katzourakis-2017-Dsols} as an independent general framework to treat fully nonlinear systems of PDEs of any order, and has been subsequently utilised in the $L^\infty$ context in a number of works (see e.g.\ \cite{Katzourakis-Ayanbayev-2017, Katzourakis-Moser-2019}). The theory of $\mathcal{D}$-solutions uses Young measures to allow for a merely measurable map to be interpreted as a solution to a PDE, by exploring weak* limits of difference quotients, the latter viewed as probability-valued maps in appropriate compactifications.

The investigation of $\mathcal{D}$-solutions to $\A^2_\infty u=0$ is necessary firstly due to the nature of the operator $\A^2_\infty$ - being fully nonlinear and not (degenerate) elliptic, it is not amenable to the extant standard notions of generalised solutions. Secondly, the natural regularity class of this second-order variational problem is $W^{2,\infty}(\Omega)$, while the operator $\A^2_\infty$ is of third-order. This is easily seen when considering the expanded form of the equation \eqref{eqn:A2inftensor}, which is expressed in the operator 
\[
\mathcal{A}_\infty\ : \ \ \Omega\times\R\times\R^n\times\R^{n^{\otimes2}}_s\times\R^{n^{\otimes3}}_s\longrightarrow \R,
\]
defined as
\begin{equation}
\label{Aronsson-operator}
  \begin{aligned}
  \mathcal{A}_\infty(x,\eta,\p,\X,\bZ):=\sum_{i,j,k,l,p,q=1} & \H_{\X_{ij}}(x,\eta,\p,\X)\Big(\H_{x_i}(x,\eta,\p,\X)+\H_{\eta}(x,\eta,\p,\X)\p_i 
  \\
  & \ \ \ \ \ \ \ +\H_{\p_k}(x,\eta,\p,\X)\X_{ki} +\H_{\X_{kl}}(x,\eta,\p,\X)\bZ_{ikl}\Big) \cdot
  \\
 & \cdot \Big(\H_{x_j}(x,\eta,\p,\X)+\H_{\eta}(x,\eta,\p,\X)\p_j 
 \\
 & \ \ \ \ \ \ \ +\H_{\p_p}(x,\eta,\p,\X)\X_{pj} +\H_{\X_{pq}}(x,\eta,\p,\X)\bZ_{jpq}\Big).
  \end{aligned}
\end{equation}
For brevity, we will symbolise $\mathcal{A}_\infty$ in compact tensor notation as
\[
  \begin{aligned}
  \mathcal{A}_\infty(x,\eta,\p,\X,\bZ)= \H_\X (x,\eta,\p,\X) \! : \!\Big(&\H_{x_i}(x,\eta,\p,\X)+\H_\eta(x,\eta,\p,\X)\p\\&+\H_\p(x,\eta,\p,\X)\X+\H_\X(x,\eta,\p,\X)\!:\!\bZ\Big)^{\otimes2},
  \end{aligned}
\]
where the exact meaning of the above is given by \eqref{Aronsson-operator}. Then, the contracted form of equation \eqref{eqn:A2inftensor} can be rewritten in expanded form as
\begin{equation} \label{expanded-PDE}
\mathcal{A}_\infty(\J^2u,\D^3u)=0 \ \ \text{ in }\Omega
\end{equation}
(and, equivalently, also as $\mathcal{A}_\infty(\J^3u)=0$). Even though $\mathcal{A}_\infty$ only makes sense, strictly speaking, when $u\in C^3(\Omega)$, it is the appropriate form in order to define and study twice weakly differentiable (generalised) $\mathcal{D}$-solutions to the Dirichlet problem for $\mathcal{A}_\infty(\J^2u,\D^3u)=0$ in $\Omega$.

Let us now outline the content of this paper. This introduction is followed by section \ref{section:preliminaries}, in which we discuss some basic elements of the theory of $\mathcal{D}$-solutions for fully nonlinear third-order PDE, to the extent they are utilised in the present paper. 

In section \ref{section:existenceofminimisers} we present the existence of global and absolute minimisers of $\E_\infty$ on $\Omega$. Existence of global minimisers is established in Theorem \eqref{theorem:existenceofglobalminimisers}, following the structure of the similar proofs in \cite{Katzourakis-Pryer-2020} and \cite{BJW-2001-Euler}, using $L^p$ approximations and Young measures. Even though the direct method applied straight to the supremal functional would work (under natural assumptions that render it sequentially weakly* lower semicontinuous and coercive), it provides \enquote{too many minimisers}, whereas the method of $L^p$ approximations seems to \enquote{select} the preferable one, and we see in Theorem \eqref{theorem:existenceofabsoluteminimisers} that this is indeed the case, at least when $n=1$. The assumptions on the supremand $\H$ are quite weak, namely a growth bound for coercivity and level-convexity in the final argument (convex sublevel sets). See the recent paper \cite{Ribeiro-Zappale-2024} for a thorough contemporary review of notions of convexity in $L^\infty$ variational problems. 

In section \ref{section:variationalcharacterisation} we (formally) derive the analogue of the Euler-Lagrange equation for $\E_\infty$, which is the third-order PDE $\A^2_\infty u=0$ defined in \eqref{eqn:A2inftensor} (rewritten in \eqref{eqn:A2infindex} and \eqref{expanded-PDE}), and is obtained by \enquote{taking the limit} of the Euler-Lagrange equation for the approximating $L^p$ functionals. Let us note that in \cite[Remark 4.9]{Aronsson-Barron} this PDE was also previously derived, but it was not studied at all, and the derivation itself did not apply as it stands when $\H$ attained negative values, as in our case. Subsequently, we consider the satisfaction of this PDE by absolute minimisers of $\E_\infty$, at least under additional regularity assumptions for $C^3$ absolute minimisers. Note that this won't necessarily hold for global minimisers as they lack the locality required. The method we follow is new and in particular offers a simpler proof of (\cite{Katzourakis-Pryer-2020}, Theorem 14(A)), even when there are no lower-order terms.

In section \ref{section:existenceofDsolutions} we focus our attention on the PDE \eqref{eqn:A2inftensor} in its own right. Without imposing any convexity assumptions on $\H$, we first establish the existence of solutions in $W^{2,\infty}(\Omega)$ to the (first-order) Dirichlet problem for the auxiliary implicit PDE given by
\begin{equation}
\label{auxiliary-problem}
  \H(\J^2u)=C\ \text{ a.e. in }\Omega ; \ \ u=g, \ \D u= \D g \text{ on } \partial \Omega,
\end{equation}
for some compatible $C>0$ and any $g\in W^{2,\infty}(\Omega)$. The solvability of \eqref{auxiliary-problem}, which is of independent mathematical interest, is based on an application of the Baire Category method (see e.g.\ \cite{Dacorogna-Marcellini-1999}). The key observation is that satisfaction of \eqref{auxiliary-problem} implies $\D(\H(\J^2u))=0$ a.e. on $\Omega$, and the latter appears as a multiplicative term in \eqref{eqn:A2inftensor}. Using this observation and some basic technical machinery from \cite{Katzourakis-2017-Dsols}, the existence of $\mathcal{D}$-solutions in $W_g^{2,\infty}(\Omega)$ for the Dirichlet problem for \eqref{eqn:A2inftensor} on $\Omega$ is then proved. In particular, we obtain a more concise streamlined method of proof in the case of no lower-order terms, in comparison to the first-principles approach used in \cite{Katzourakis-Pryer-2020}. Let us finally close this introduction by noting that further results for higher-order supremal variational problems through different means and with a different focus which do not involve the PDE \eqref{eqn:A2inftensor} have recently been obtained in \cite{Katzourakis-Moser-2019, Katzourakis-Moser-2023, Katzourakis-Moser-2024-LocalMinimisers}.

%%%%%%%%%%%%%%%%%%%%%%%%%%%%%%%%%%%%%%%%%%%%%%%%%%%%%%%%%%%%%%

\smallskip

\section{Preliminaries}\label{section:preliminaries}

In this section we recall, for the convenience of the reader, some well-known results on Young measures. These are required in order to define $\mathcal{D}$-solutions, as well as for some of our proofs involving non-convex variational problems. Young measures valued into the Euclidean space $\R^{n^{\otimes2}}_s$ are used in the proof of existence of global minimisers, and Young measures valued into the compactified space $\overbar{\R}^{n^{\otimes3}}_s$ are used when looking at $\mathcal{D}$-solutions to capture the limiting behaviour of possibly escaping difference quotients. The latter space is defined as the 1-point compactification of $\R^{n^{\otimes3}}_s$, through $\overbar{\R}^{n^{\otimes 3}}_s:=\R^{n^{\otimes 3}}_s\cup\{\infty\}$. We refer e.g.\,to \cite{Fonseca-Leoni-2007} for Euclidean Young measures and \cite{Florescu-Godet-Thobie-2012} for Young measures valued into compact spaces. 

Recall the dual of $C(\overbar{\R}^{n^{\otimes3}}_s)$ is the space of signed Radon measures $\mathcal{M}(\overbar{\R}^{n^{\otimes3}}_s)$.
\begin{defn}
[Young measures]
\label{defn:youngmeasures}
The set $\mathscr{Y}(\Omega,\overbar{\R}^{n^{\otimes3}}_s)$ of Young measures valued into the compact space $\overbar{\R}^{n^{\otimes3}}_s$ consists of weakly* measurable maps $x\mapsto\vartheta(x)$ for a.e. $x\in\Omega$, where $\vartheta$ is a probability measure. This set $\mathscr{Y}(\Omega,\overbar{\R}^{n^{\otimes3}}_s)$ is a subset of the unit sphere of $L^\infty_{w^*}(\Omega,\mathcal{M}(\overbar{\R}^{n^{\otimes3}}_s))$, which in turn is the dual of $L^1(\Omega,C(\overbar{\R}^{n^{\otimes3}}_s))$. The associated duality pairing is:
\begin{equation}
  \braket{\vartheta,\Phi}:=\int_\Omega\int_{\overbar{\R}^{n^{\otimes3}}_s}\Phi(x,\bZ)\hspace{1mm}\d[\vartheta(x)](\bZ)\hspace{1mm}\d x,
\end{equation} 
where $\vartheta\in L^\infty_{w^*}(\Omega,\mathcal{M}(\overbar{\R}^{n^{\otimes3}}_s))$ and $\Phi\in L^1(\Omega,C(\overbar{\R}^{n^{\otimes3}}_s))$. 
\end{defn}
Every measurable map $v:\Omega\to\overbar{\R}^{n^{\otimes3}}_s$ generates a Young measure $\delta_v\in\mathscr{Y}(\Omega,\overbar{\R}^{n^{\otimes3}}_s)$ given by $\delta_{v(x)}=:\delta_v(x)$. Let $(v_i)_1^\infty$, $v_i:\Omega\to\R^{n^{\otimes3}}_s$, be a sequence of measurable functions. We have $\delta_{v_i}\weakstar\delta_{v_\infty}$ in $\mathscr{Y}(\Omega,\overbar{\R}^{n^{\otimes3}}_s)$ if and only if $v_i(x)\longconv v_\infty(x)$ for a.e. $x\in\Omega$ (possibly along a subsequence), for some $v_\infty$ measurable. Furthermore, the set $\mathscr{Y}(\Omega,\overbar{\R}^{n^{\otimes3}}_s)$ is sequentially weakly* compact and convex. This means, even without any pointwise convergence assumption, there always exists some $\vartheta\in\mathscr{Y}(\Omega,\overbar{\R}^{n^{\otimes3}}_s)$ such that $\delta_{v_i}\weakstar\vartheta$ along a subsequence as $i\to\infty$. We can now define diffuse derivatives as an essential object in defining in $\mathcal{D}$-solutions.
\begin{defn}[Difference quotients and diffuse derivatives]\label{defn:diffusederivatives}
The difference quotient of a function $v\in L^1_{\mathrm{\mathrm{loc}}}(\Omega)$, for $h\neq 0$, is given by 
\begin{equation}
  \begin{aligned}
  \D^{1,h}_k v(x)&:=\frac{1}{h}\big(v(x+he^k)-v(x)\big), \\
  \D^{1,h}v&:=\big(\D^{1,h}_1v,\ldots,\D^{1,h}_nv\big),
  \end{aligned}
\end{equation}
where $e^k$ is the $k$-th standard basis vector in $\R^n$. Consider now $u\in W^{2,1}_{\mathrm{loc}}(\Omega)$, and take $v:=\D^2u:\Omega\to\R^{n^{\otimes2}}_s$. The diffuse third-order derivatives $\mathcal{D}^3u\in\mathscr{Y}(\Omega,\overbar{\R}^{n^{\otimes3}}_s)$ are obtained as the weak* subsequential limits of the difference quotients $\D^{1,h}\D^2u$ in $\mathscr{Y}(\Omega,\overbar{\R}^{n^{\otimes3}}_s)$ along infinitesimal sequences $(h_m)_1^\infty$, which is to say $\delta_{\D^{1,h_m}\D^2u}\weakstar\mathcal{D}^3u$ in $\mathscr{Y}(\Omega,\overbar{\R}^{n^{\otimes3}}_s)$, as $m\to\infty$.
\end{defn}

The sequential weak* compactness of the space of Young measures implies that any function has diffuse derivatives of all orders. We can now make rigorous the notion of twice (weakly) differentiable $\mathcal{D}$-solutions of the third-order PDE $\A^2_\infty u=0$ in its expanded form $\mathcal{A}_\infty(\J^2u,\D^3u)=0$. For technical reasons dictated by the proof of Theorem \ref{theorem:existenceofDsolutions}, we actually state it more generally for arbitrary fully nonlinear third-order systems.

\begin{defn}
[Twice weakly differentiable $\mathcal{D}$-solutions of a third-order PDE]
\label{D-solutions}
Let $N\in \mathbb N$, and let $\mathcal{F}:\Omega\times \R\times \R\times \R^{n^{\otimes2}}_s \! \times \R^{n^{\otimes3}}_s\to\R^N$ be a Borel measurable mapping. We say that $u\in W^{2,1}_{\mathrm{loc}}(\Omega)$ is a $\mathcal{D}$-solution of the PDE system 
\[
\mathcal{F}(\J^2u,\D^3u)=0 \ \ \text{ in }\Omega,
\]
if for any diffuse third-order derivative $\mathcal{D}^3u\in\mathscr{Y}(\Omega,\overbar{\R}^{n^{\otimes3}}_s)$, we have that 
  \[
    \int_{\overbar{\R}^{n^{\otimes3}}_s}\Phi(\bZ)\mathcal{F}(\J^2u(x),\bZ)\hspace{1mm}\d[\mathcal{D}^3u(x)](\bZ)=0,
  \]
  for a.e. $x\in\Omega$, and for any test function $\Phi\in C_c({\R}^{n^{\otimes3}}_s)$.
\end{defn}
It can easily be seen that $\mathcal{D}$-solutions are compatible with other pointwise notions of solution. In particular, if the Hessian of the $\mathcal{D}$-solution is differentiable in measure (or is $C^3$), then we have a strong a.e.\ solution on $\Omega$ (or classical). Conversely, any thrice in measure differentiable strong solution is a $\mathcal{D}$-solution to the same equation. 

We close this section with the following equivalent formulation of Definition \ref{D-solutions} that will be utilised in the proof of  Theorem \ref{theorem:existenceofDsolutions}.

\begin{rmk} \label{remark2.4} In view of \cite[Proposition 21]{Katzourakis-2017-Dsols}, we have the following restatement of Definition \ref{D-solutions}: for any diffuse third-order derivative $\mathcal{D}^3u\in\mathscr{Y}(\Omega,\overbar{\R}^{n^{\otimes3}}_s)$, we have
 \[
 \sup_{\bZ \in \mathrm{supp}_*(\mathcal{D}^3u(x))}  \big| \mathcal{F}(\J^2u(x),\bZ) \big| =0, \ \ \text{ a.e. }x\in\Omega,
  \]
where ``\,$\mathrm{supp}_*$\!" is the reduced support of the Young measure away from the point at infinity of the compactification of $\bar{\R}^{n^{\otimes3}}_s$:
\[
\mathrm{supp}_*(\mathcal{D}^3u(x)) : = \mathrm{supp}(\mathcal{D}^3u(x)) \cap {\R}^{n^{\otimes3}}_s.
\]
\end{rmk}
For more details on the properties and structure of $\mathcal{D}$-solutions we refer to \cite{Katzourakis-2017-Dsols}.

%%%%%%%%%%%%%%%%%%%%%%%%%%%%%%%%%%%%%%%%%%%%%%%%%%%%%%%%%%%%%%

\smallskip

\section{Existence of Minimisers}\label{section:existenceofminimisers}

In this section we present a result on the existence of global minimisers of the functional \eqref{eqn:fullfunctional} given arbitrary (first-order) Dirichlet boundary conditions on a fixed open and bounded domain $\Omega\subseteq\R^n$. Our assumptions require only a coercivity lower-bound and level-convexity of $\H$ in the final (Hessian) argument, which is to say that the sublevel sets be convex. Even though in itself this can be seen as a simple application of the direct method by utilising the weak* lower-semicontinuity of $\E_\infty$ in $W^{2,\infty}(\Omega)$, instead, following \cite{Katzourakis-Pryer-2020}, we adopt a more complicated approach and construct a ``good" minimiser by using the method of $L^p$ approximations as $p\to\infty$. The reason for this is that, as it is well known, in general $L^\infty$ minimisers are non-unique, but the method of $L^p$ approximations allows to select ``the best one" which satisfies additional properties and can usually be shown to be absolute. This is indeed the case in our work as well, at least when $n=1$.

We note, however, that level-convexity is a weaker condition than Morrey's quasiconvexity, which is necessary to ascertain that the approximating sequences of $L^p$ functionals attain their infima for all $p$ finite (up to extra growth bounds on the integrand). Thus, inspired by \cite{BJW-2001-Euler}, we pass through approximate minimisers in the $L^p$ approximating sequence utilising Young measures, and invoking the extended Jensen inequality \eqref{eqn:extendedjenseninequality}. Since this only holds in the Hessian argument, some care has to be taken to account for the inclusion of lower-order terms.
A (diagonal) lower semicontinuity of $\E_\infty$, which is useful for Theorem \ref{theorem:existenceofabsoluteminimisers}, is then shown in a separate result. We then turn to the existence of absolute minimisers, and prove that, if $n=1$, the minimiser obtained through $L^p$ approximations is in fact an absolute minimiser. 

We now present the main result of the section. Since $\H$ will not be assumed to be non-negative, the integral functional \eqref{1.8} cannot be used directly to approximate \eqref{eqn:fullfunctional} as $p\to\infty$. Instead, we will use the modified approximations
\begin{equation}
\label{3.1}
  \E_p(u,\O):=\bigg(\av_{\O} \Big(M + \H(x,u(x),\D u(x),\D^2u(x)) \Big)^p\hspace{1mm}\d x\bigg)^{\frac1p} -M,
\end{equation} 
which works as long as we have a lower bound for the form $\H \geq -m$, for some $m>0$, by choosing any fixed $M>m$.

%%%%%%%%%%%%%%%%%%%%%%%%%%%%%%%%%%%%%%%%%%%%%%%%%%%%%%%%%%%%%%

\begin{thm}[Existence of global minimisers]\label{theorem:existenceofglobalminimisers}
Let $\Omega\Subset\R^n$ be open, bounded with Lipschitz boundary $\partial \Omega$, and $n\in\mathbb{N}$. Let $\H : \Omega\times \big(\R\times\R^n\times\R^{n^{\otimes2}}_s) \to \R$ be a Carath\'eodory function, which is bounded below and level-convex in its final argument, namely for all $(x,\eta,\mathrm p)\in\Omega\times\R\times\R^n$ and $t\in \R$, the set $\{\H(x,\eta,\mathrm p,\cdot)\le t\}$ is convex. We further suppose that $\H$ satisfies the following coercivity condition: there exist $C_1,C_2>0$ and $0\le s, t< 1$ such that
\begin{equation}
\label{coercivity}
\H(x,\eta,\mathrm p, \X)\ge C_1\lvert \X\rvert - C_2(1+\lvert \eta\rvert^s+\lvert \mathrm p\rvert^t)
\end{equation}
for a.e.\ $x\in \Omega$ and all $(\eta,\mathrm p, \X)\in  \R \times \R^n \times \R^{n^{\otimes2}}_s$. We also assume the following continuity assumption holds true: for any $R>0$, there exists $\ell=\ell(R)>1$ and an increasing modulus of continuity $\omega_R \in C[0,\infty)$ with $\omega_R(0)=0$, such that
\begin{equation}
\label{continuity}
\Big|\H(x,\eta',\mathrm p', \X) - \H(x,\eta'',\mathrm p'', \X) \Big| \leq
\omega_{R}\Big(|\eta'-\eta''|+|\mathrm p' - \mathrm p''| \Big)(1+|\X|^\ell)
\end{equation}
for a.e.\ $x\in \Omega$, any $\eta',\eta'' \in \mathbb (-R,R)$, any $\mathrm p',\mathrm p''  \in \mathbb B_R(0)$, and any $\X \in \R^{n^{\otimes2}}_s$.

Then, for any boundary data $g\in W^{2,\infty}(\Omega)$, there exists a global minimiser $u_\infty\in W^{2,\infty}_g(\Omega)$ of the functional \eqref{eqn:fullfunctional} on $\Omega$. Furthermore, for any fixed $q \in (1,\infty)$, $u_\infty$ is the subsequential weak $W^{2,q}(\Omega)$-limit of approximate minimisers of the functionals $\{\E_p:p\in(1,\infty)\}$ as $p\to \infty$.
\end{thm}

%%%%%%%%%%%%%%%%%%%%%%%%%%%%%%%%%%%%%%%%%%%%%%%%%%%%%%%%%%%
\begin{rmk}[More general supremands] We note that, even though the assumptions are rather natural and non-restrictive, Theorem \ref{theorem:existenceofglobalminimisers} applies to a much wider class of supremands $\H$ than those it appears to. Indeed, Theorem \ref{theorem:existenceofglobalminimisers} applies also to 
\[
\ \ \ \ \ \ \mathrm F_\infty (u,\O) :=\, \underset{\O}{\mathrm{ess}\sup}\hspace{1mm} \Phi\big(\H(\J^2u)\big),  \ \ \ \  u\in W^{2,\infty}(\Omega), \ \O\subseteq\Omega \text{ measurable},
\]
where $\Phi : \R \to \R$ is \emph{any} strictly increasing lower semicontinuous function, (possibly discontinuous, non-differentiable, and of arbitrary growth). This is based on the observation that $\Phi$ commutes with the (essential) supremum, therefore $\E_\infty$ and $\mathrm F_\infty = \Phi \circ \E_\infty$ have the same sets of minimisers (and absolute minimisers).
\end{rmk}

A key ingredient for the proof is the extended Jensen inequality (see \cite{BJW-2001-Euler}), which, adapted to our setting states that for any level-convex continuous function $F:\R^{n^{\otimes2}}_s\to\R$, and any probability measure $\vartheta \in \mathcal{P}(\R^{n^{\otimes2}}_s)$, we have
\begin{equation}
\label{eqn:extendedjenseninequality}
F\bigg(\int_{\R^{n^{\otimes2}}_s}\X\hspace{1mm}\d\vartheta(\X)\bigg)
\leq
\vartheta-\underset{\X\in\R^{n^{\otimes2}}_s}{\mathrm{ess}\sup}\hspace{1mm}F(\X).
\end{equation}

%%%%%%%%%%%%%%%%%%%%%%%%%%%%%%%%%%%%%%%%%%%%%%%%%%%%%%%%%%%

\begin{proof}[\rm \textbf{Proof of Theorem} \ref{theorem:existenceofglobalminimisers}]\label{proof:existenceofglobalminimisers}
Fix $p>1+1/s+1/t$. Let $(u_{p,i})_{i=1}^\infty\subseteq W^{2,\infty}_g(\Omega)$ be a minimising sequence of the integral functional $\E_p(\cdot,\Omega)$ satisfying
\[
  \E_p(u_{p,i},\Omega)\to\inf\big\{\E_p(v,\Omega):v\in W^{2,\infty}_g(\Omega) \big\}
\]
as $i\to\infty$. Since minimisers of $\E_p$ may not exist, we consider instead approximate minimisers by selecting $i=i(p)$ large enough such that 
\[
  \E_p(u_p,\Omega)\le 2^{-p}+\inf\big\{\E_p(v,\Omega):v\in W^{2,\infty}_g(\Omega)\big\},
\]
where $u_p:=u_{p,i(p)}$. By applying the H\"older inequality, we have 
\begin{equation}\label{eqn:Equpperboundwithp}
  \E_q(u_p,\Omega)\le2^{-p}+\E_\infty(\psi,\Omega),
\end{equation}
for any $\psi\in W^{2,\infty}_g(\Omega)$, when $q\le p$. In particular, this implies 
\begin{equation}
\label{3.7A}
  \E_p(u_p,\Omega)\le 1+\E_\infty(g,\Omega).
\end{equation}
We now seek to bound the sequence $(u_p)_1^\infty$ in $W^{2,k}(\Omega)$, where $1<k<\infty$. By assumption \eqref{coercivity}, for any $v\in W^{2,\infty}(\Omega)$, we have 
\[
\av_{\Omega}\big( M+ \H(\J^2v)\big)^k\ge\frac{1}{2^{k-1}}C^k_1\av_{\Omega}\lvert\D^2v\rvert^k-3^{k-1}C_2^k\bigg(1+\av_{\Omega}\lvert v\rvert^{sk}+\av_{\Omega}\lvert\D v\rvert^{tk}\bigg).
\]
Since $0\leq s,t<1$, by applying the H\"older inequality, we obtain 
\[
\begin{split}
  \bigg(\av_{\Omega}\big( M+ \H(\J^2v)\big)^k\bigg)^\frac{1}{k}
  & \ge \,
  C_1\bigg(\av_{\Omega}\lvert\D^2v\rvert^k\bigg)^\frac{1}{k}-C_2\bigg[1+\bigg(\av_{\Omega}\lvert v\rvert^{sk}\bigg)^\frac{1}{k}+\bigg(\av_{\Omega}\lvert\D v\rvert^{tk}\bigg)^\frac{1}{k}\bigg]
  \\
  & \ge \,
  C_1\bigg(\av_{\Omega}\lvert\D^2v\rvert^k\bigg)^\frac{1}{k}-C_2\bigg[1+\bigg(\av_{\Omega}\lvert v\rvert^{k}\bigg)^\frac{s}{k}+\bigg(\av_{\Omega}\lvert\D v\rvert^{k}\bigg)^\frac{t}{k}\bigg].
  \end{split}
\]
In particular, for any $p\geq k$, using the H\"older inequality again, the above estimate implies
\begin{equation}
\label{coercivity-1}
\E_p(v,\Omega)
 \ge 
 \, C_1\bigg(\av_{\Omega}\lvert\D^2v\rvert^k\bigg)^\frac{1}{k}-C_2\bigg[1+\bigg(\av_{\Omega}\lvert v\rvert^{k}\bigg)^\frac{s}{k}+\bigg(\av_{\Omega}\lvert\D v\rvert^{k}\bigg)^\frac{t}{k}\bigg] -M.
\end{equation}
Suppose now that $v\in W^{2,\infty}_g(\Omega)$. We estimate
\begin{equation}
\label{3.9A}
\begin{split}
  \E_p(v,\Omega)
  & \ge \,
  C_1\bigg(\av_{\Omega}\big| \D^2v-\D^2g\big|^k\bigg)^\frac{1}{k}
  \\
  & \ \ \ \ -C_2\bigg[1 + \bigg(\av_{\Omega}\lvert v-g\rvert^{k}\bigg)^\frac{s}{k}+\bigg(\av_{\Omega}\lvert\D v-\D g)\rvert^{k}\bigg)^\frac{t}{k}\bigg]
  \\
  & \ \ \ \ - (C_1+C_2) \bigg[  \bigg(\av_{\Omega}\lvert\D^2g\rvert^k\bigg)^\frac{1}{k}+\bigg(\av_{\Omega}\lvert g\rvert^{k}\bigg)^\frac{s}{k}+\bigg(\av_{\Omega}\lvert\D g\rvert^{k}\bigg)^\frac{t}{k}\bigg] -M.
  \end{split}
\end{equation}
Since $v-g \in W^{2,\infty}_0(\Omega)$, by repeated applications of the Poincar\'e inequality, there exists $C_3=C_3(k)>1$ (depending also on $\Omega$) such that
\begin{equation}
\label{3.10A}
 (C_3(k)-1)\bigg(\av_{\Omega}\big| \D^2v-\D^2g\big|^k\bigg)^\frac{1}{k} \geq \bigg(\av_{\Omega}\big| \D v-\D g\big|^k\bigg)^\frac{1}{k} + \bigg(\av_{\Omega}\big| v- g\big|^k\bigg)^\frac{1}{k} .
\end{equation}
By combining \eqref{3.9A}-\eqref{3.10A}, we obtain
\[
\begin{split}
  \E_p(v,\Omega)
  & \ge \,
  \frac{C_1}{C_3(k)}\| v-g \|_{W^{2,k}(\Omega)}-C_2\Big[1+\big(\| v - g \|_{W^{2,k}(\Omega)}\big)^s + \big(\| v -g\|_{W^{2,k}(\Omega)}\big)^t \Big] 
  \\
  & \ \  \ \ - 2(C_1+C_2)\big(2+ \| g \|_{W^{2,k}(\Omega)}\big) -M.
  \end{split}
\]
The above estimate implies
\begin{equation}
\label{3.9}
\begin{split}
  \E_p(v,\Omega)
  & \ge \,
  \frac{C_1}{C_3(k)}\| v\|_{W^{2,k}(\Omega)}-C_2\Big[1+\big(\| v \|_{W^{2,k}(\Omega)}\big)^s + \big(\| v\|_{W^{2,k}(\Omega)}\big)^t \Big] 
  \\
  & \ \  \ \ - \bigg[2C_1+3C_2+\frac{C_1}{C_3(k)}\bigg]\big(2+ \| g \|_{W^{2,k}(\Omega)}\big) -M,
  \end{split}\end{equation}
for any $v\in W^{2,\infty}_g(\Omega)$, and any fixed $k \in (1,\infty)$ with $p\geq k$. Consider now the real function
\[
f \ :\ \ [0,\infty) \to \R,\ \ \  f(\lambda):=\frac{C_1}{C_3(k)} \lambda-C_2(1+\lambda^s+\lambda^t).
\]
Since $f''>0$ on $(0,\infty)$, the function is convex, and by the mean value theorem there exists $\lambda_0 >0$ such that $f(\lambda_0)=0$ and the graph of $f$ lies above its tangent at this root, whilst also $f'(\lambda_0)>0$. Therefore, 
\begin{equation}
\label{3.10}
\begin{aligned}
  \frac{C_1}{C_3(k)} \lambda-C_2(1+ \lambda^{s}+ \lambda^{t}) &=f(\lambda)
  \\
   & \ge f(\lambda_0)+f'(\lambda_0)(\lambda-\lambda_0)
  \\
  &= f'(\lambda_0)\lambda- f'(\lambda_0) \lambda_0 
  \\
  &=: A(k)\lambda-B(k),
\end{aligned}
\end{equation}
where $A(k),B(k)>0$ are constants depending on $k$ through $C_3(k)$, as well as on $s, t, C_1, C_2$. By combining \eqref{3.9} with \eqref{3.10}, we deduce
\begin{equation}
\label{coercivity-2}
\begin{split}
  \E_p(v,\Omega)
  & \ge \,
  A(k)\| v \|_{W^{2,k}(\Omega)}
  \\
  &\ \ \ \ - \bigg\{B(k) + \bigg[ 2C_1+3C_2+\frac{C_1}{C_3(k)}\bigg]\big(2+ \| g \|_{W^{2,k}(\Omega)}\big) + M \bigg\},
    \end{split}
\end{equation}
for any $v\in W^{2,\infty}_g(\Omega)$, and any fixed $k \in (1,\infty)$ with $p\geq k$. Selecting $v:=u_p$ in \eqref{coercivity-2}, we infer that $(u_p)_{p>1}$ is bounded in $W^{2,k}(\Omega)$, for all $k \in (1,\infty)$. Since $W^{2,k}(\Omega)$ is reflexive for $1<k<\infty$, we have 
$u_p\weaklyconv u_\infty$  in $W^{2,k}(\Omega)$ as $p\to\infty$, along a subsequence. Further, by uniqueness of weak limits, $u_\infty\in \bigcap_{1<k<\infty}W^{2,k}(\Omega)$. Additionally, since $W^{2,k}(\Omega) \Subset C^{1,\alpha}(\overline{\Omega})$, by passing to a further subsequence if necessary, we have $u_p \longrightarrow u_\infty$ strongly in $C^{1,\alpha}(\overline{\Omega})$ for all $\alpha\in(0,1)$. We now return to \eqref{coercivity-1} with $v=u_p$. Using the weak lower semicontinuity of the $L^k$-norm, and noting that $u_p \longrightarrow u_\infty$ in $C^{1}(\overline{\Omega})$ and $\D^2u_p\weaklyconv \D^2u_\infty$ in $L^{k}(\Omega;\R^{n^{\otimes 2}}_s)$ as $p\to \infty$ along a subsequence, we deduce by \eqref{3.7A} that
\[
\bigg(\av_{\Omega}\lvert\D^2u_\infty\rvert^k\bigg)^\frac{1}{k} \leq \frac{1}{C_1}\bigg\{1 + \E_\infty(g,\Omega)+ C_2\bigg[1+\bigg(\av_{\Omega}\lvert u_\infty\rvert^{k}\bigg)^\frac{s}{k}+\bigg(\av_{\Omega}\lvert\D u_\infty\rvert^{k}\bigg)^\frac{t}{k}\bigg]\bigg\}.
\]
By letting $k\to \infty$, we infer that $\D^2 u_\infty \in L^\infty(\Omega;\R^{n^{\otimes 2}}_s)$, whence $u_\infty \in W^{2,\infty}(\Omega)$. Consider now the Young measures 
\[
(\delta_{\D^2u_p})_{p>1}\subseteq\mathscr{Y}(\Omega,\R^{n^{\otimes2}}_s)
\]
generated by the Hessians $\D^2u_p\in L^p(\Omega;\R^{n^{\otimes 2}}_s)$. Since $(\D^2u_{p})_{p>1}$ is bounded in $L^k(\Omega;\R^{n^{\otimes 2}}_s)$ for all $k\in(1,\infty)$, it is equi-integrable and tight, therefore by passing perhaps to a further subsequence there exists a Young measure $\vartheta_\infty\in\mathscr{Y}(\Omega,\R^{n^{\otimes2}}_s)$ with 
\[
\delta_{\D^2u_p}\weakstar \vartheta_\infty \ \ \text{ in }\mathscr{Y}(\Omega,\R^{n^{\otimes2}}_s), 
\]
as $p\to\infty$. Further, since $\D^2 u_\infty \in L^\infty(\Omega;\R^{n^{\otimes 2}}_s)$, it follows that $\vartheta_\infty$ has compact support in $\R^{n^{\otimes2}}_s$, which is to say there exists some $R>0$ such that $\vartheta_\infty(x)\subseteq\mathbb{B}_R(0)$ for a.e.\ $x\in\Omega$, where $\mathbb{B}_R(0)$ is the ball of radius $R$ centred at the origin in $\R^{n^{\otimes2}}_s$. Additionally, $\vartheta_\infty$ has barycentre $\D^2u_\infty$, which is to say 
\[
  \D^2u_\infty(x)=\int_{\R^{n^{\otimes2}}_s}\X\hspace{1mm}\d[\vartheta_\infty(x)](\X)\ \ \text{ for a.e. }x\in\Omega.
\]
Now, since $\H$ is level-convex in the last variable, by the extended Jensen inequality \eqref{eqn:extendedjenseninequality} we have 
\[
\begin{split}
\H\big(\J^2u_\infty(x)\big) &=  \H\bigg(\J^1u_\infty(x),\int_{\R^{n^{\otimes2}}_s}\X\hspace{1mm}\d[\vartheta_\infty(x)](\X)\bigg)
\\
&\le [\vartheta_\infty(x)]-\underset{ \X \in \R^{n^{\otimes2}}_s}{\mathrm{ess}\sup}\hspace{1mm}\H(\J^1u_\infty(x),\X) ,
\end{split}
\]
for a.e.\ $x\in\Omega$. Hence, we may estimate
\[
\begin{aligned}
\underset{ \Omega}{\mathrm{ess}\sup}\hspace{1mm} \H (\J^2u_\infty) & =
  \underset{x\in \Omega}{\mathrm{ess}\sup}\hspace{1mm} \big(M+\H (\J^2u_\infty(x))\big) -M
  \\
  & \leq  \underset{ x\in \Omega}{\mathrm{ess}\sup}\hspace{1mm}\bigg(
  [\vartheta_\infty(x)]-\underset{\X \in \R^{n^{\otimes2}}_s}{\mathrm{ess}\sup}\hspace{1mm}\Big(M+\H\big(\J^1u_\infty(x),\X \big)\Big)\bigg) -M.
\end{aligned}
\]
Therefore,
\[
\begin{aligned}
\underset{ \Omega}{\mathrm{ess}\sup}\hspace{1mm} \H (\J^2u_\infty)  
  &\le \lim_{q\to\infty}\bigg(\av_{\Omega}\int_{\R^{n^{\otimes2}}_s}\big(M+\H(\J^1u_\infty(x) ,\X)\big)^q\hspace{1mm}\d[\vartheta_\infty(x)](\X)\hspace{1mm}\d x \bigg)^{\!\! \frac1q} -M
  \\
  &\le \liminf_{q\to\infty}\bigg[\liminf_{p\to\infty}\bigg(\av_{\Omega}\int_{\R^{n^{\otimes2}}_s}\big(M+\H(\J^1u_\infty ,\X)\big)^q\hspace{1mm}\d[\delta_{\D^2u_p} ](\X) \bigg)^{\!\! \frac1q}\bigg] -M
  \\
   &=\liminf_{q\to\infty}\bigg[\liminf_{p\to\infty}\bigg(\av_{\Omega} \big(M+\H(\J^1u_\infty ,\D^2u_p)\big)^q \bigg)^{\!\! \frac1q}\bigg] -M.
\end{aligned}
\]
Since $u_p \longrightarrow u_\infty$ in $C^1(\overline{\Omega})$ as $p\to\infty$ along a sequence $(p_j)_1^\infty$, there exists $R>0$ such that
\[
u_\infty(\Omega)\bigcup_{j\in\mathbb N} u_{p_j}(\Omega) \subseteq (-R,R),  \ \ \D u_\infty(\Omega)\bigcup_{j\in\mathbb N} \D u_{p_j}(\Omega) \subseteq \mathbb B_R(0).
\]
By assumption \eqref{continuity}, there exists $\ell>1$ and a modulus of continuity $\omega_R \in C[0,\infty)$ such that
\[
\begin{aligned}
  \lvert\H(\J^1u_\infty,\D^2u_p)\rvert&=\lvert\H(\J^1u_\infty,\D^2u_p)-\H(\J^1u_p,\D^2u_p)+\H(\J^1u_p,\D^2u_p)\rvert \\&\le\lvert\H(\J^1u_\infty,\D^2u_p)-\H(\J^1u_p,\D^2u_p)\rvert+\lvert\H(\J^2u_p)\rvert \\
  &\le\omega_R\big(\lvert u_\infty-u_p\rvert+\lvert \D u_\infty-\D u_p\rvert \big)(1+ |\D^2u_p|^\ell )+\lvert\H(\J^2u_p)\rvert,
\end{aligned}
\]
a.e.\ on $\Omega$. Let us define now the function $\Gamma : \Omega\times\R\times\R^n\times\R^{n^{\otimes2}}_s\longrightarrow \R$ by setting
\[
\Gamma(x,\eta,\mathrm p, \X):=\omega_R\Big( |u_\infty(x)-\eta| + \big|\D u_\infty(x)- \mathrm p\big| \Big)(1+ |\X|^\ell).
 \]
This allows us to write
\[
  |\H(\J^1u_\infty,\D^2u_p)| \leq 
  \Gamma(\J^2 u_p) + | \H(\J^2u_p)|, 
\] 
and, in view of the above and \eqref{eqn:Equpperboundwithp}, we deduce
 \[
\begin{aligned}
\E_\infty(u_\infty,\Omega) & =\smash{\underset{ \Omega}{\mathrm{ess}\sup}\hspace{1mm} \H (\J^2u_\infty)    }
\\
& \le\liminf_{q\to\infty}\bigg[\liminf_{p\to\infty}\bigg[\bigg(\av_{\Omega}\Gamma(\J^2u_p)^q\bigg)^{\!\! \frac1q}+\bigg(\av_{\Omega}\big(M+\H(\J^2u_p)\big)^q\bigg)^{\!\! \frac1q}\bigg]\bigg] -M, 
\\
&
\le\liminf_{q\to\infty}\bigg\{\liminf_{p\to\infty}\bigg[ \omega_R\Big(\| u_p-u_\infty\|_{W^{1,\infty}(\Omega)} \Big)\bigg(\av_{\Omega}\big(1+|\D^2u_p|^{\ell}\big)^q\bigg)^{\!\! \frac1q}
\\
&\ \ \  +\bigg(\av_{\Omega}\big(M+\H(\J^2u_p)\big)^q\bigg)^{\!\! \frac1q}\bigg]\bigg\} -M, 
\\
&\le\liminf_{q\to\infty}\Big[\liminf_{p\to\infty}\big(2^{-p}+\E_\infty(\psi,\Omega)\big)\Big], 
\\
&=\E_\infty(\psi,\Omega),
\end{aligned}
\]
for any $\psi \in W^{2,\infty}_g(\Omega)$. This establishes that $u_\infty$ is a global minimiser, which completes the proof.
\end{proof}

%%%%%%%%%%%%%%%%%%%%%%%%%%%%%%%%%%%%%%%%%%%%%%%%%%%%%%%%%%%

The next result will be useful for the proof of the existence of absolute minimisers.

\begin{cor}[Diagonal lower semicontinuity]\label{corollary:diagonallowersemicontinuity}
Let $\mathcal{O}\subseteq\Omega$ be a measurable set of positive measure. If $\H\in C(\Omega\times\R\times\R^n\times\R^{n^{\otimes2}})$ satisfies the assumptions of Theorem \ref{theorem:existenceofglobalminimisers}, then the functional $\E_\infty$ is (diagonally) weakly lower semicontinuous in the following sense: 
$$
\E_\infty(u_\infty,\O)\le\liminf_{j\to\infty}\bigg(\av_{\O}\big(M+\H(\J^2u_{p_j})\big)^{p_j}\bigg)^{\frac{1}{p_j}} -M,
$$ 
where $(u_{p_j})_{j=1}^\infty$ is the subsequence along which $u_{p}\weaklyconv u_\infty$ in $W^{2,q}(\Omega)$, for all $q\in(1,\infty)$.
\end{cor}

%%%%%%%%%%%%%%%%%%%%%%%%%%%%%%%%%%%%%%%%%%%%%%%%%%%%%%%%%%%

\begin{proof}[\rm \textbf{Proof of Corollary }\ref{corollary:diagonallowersemicontinuity}]
  This proof follows the exact same structure as \cite[Lemma 5.1]{Katzourakis-2015-AbsMin}, so the details are omitted.
\end{proof}

%%%%%%%%%%%%%%%%%%%%%%%%%%%%%%%%%%%%%%%%%%%%%%%%%%%%%%%%%%%

We now present a result on the existence of absolute minimisers in dimension one.

\begin{thm}[Existence of absolute minimisers in 1D]\label{theorem:existenceofabsoluteminimisers}
Let $\Omega\subseteq\R$ be open and bounded. If $\H\in C(\Omega\times\R\times\R\times\R)$ satisfies the assumptions of Theorem \ref{theorem:existenceofglobalminimisers} with $n=1$, then the minimiser constructed therein is an absolute minimiser of $\E_\infty$ on $\Omega$, i.e. 
\begin{equation}
  \E_\infty(u,\Omega')\le\E_\infty(u+\phi,\Omega'),\hspace{1mm}\forall\hspace{1mm}\Omega'\Subset\Omega,\hspace{1mm}\forall\hspace{1mm}\phi\in W^{2,\infty}_0(\Omega').
\end{equation}
\end{thm}

%%%%%%%%%%%%%%%%%%%%%%%%%%%%%%%%%%%%%%%%%%%%%%%%%%%%%%%%%%%

This proof, which is very similar to that in \cite{Katzourakis-Pryer-2020}, makes essential use of the topology of $\R$. We do nonetheless provide it for the convenience of the reader.

%%%%%%%%%%%%%%%%%%%%%%%%%%%%%%%%%%%%%%%%%%%%%%%%%%%%%%%%%%%

\begin{proof}[\rm \textbf{Proof of Theorem }\ref{theorem:existenceofabsoluteminimisers}]
We begin with a general observation. Given $A,B,A',B',a,b\in\R$ with $a<b$, there exists a unique cubic Hermite interpolant $Q:\R\to\R$ satisfying $$Q(a)=A,Q(b)=B,Q'(a)=A',Q'(b)=B'.$$ Further, let $(v_p)_{p=1}^\infty\subseteq W^{2,\infty}(a,b)$ be any sequence of functions satisfying $v_p\to v_\infty$ in $C^1[a,b]$ as $p\to\infty$. If $Q_p$ is the cubic polynomial such that $Q_p-v_p\in W^{2,\infty}_0(a,b)$ for $p\in\mathbb{N}\cup\{\infty\}$, namely 
\begin{equation}
Q_p(a)=v_p(a),Q_p(b)=v_p(b),Q_p'(a)=v_p'(a),Q_p'(b)=v_p'(b),
\end{equation}
then we have $Q_p\longconv Q_\infty$ strongly in $W^{2,\infty}(a,b)$ as $p\to\infty$. Let $(u_p)_{p=1}^\infty$ be the sequence of approximate minimisers of Theorem \eqref{theorem:existenceofglobalminimisers} which satisfies $u_p\weaklyconv u_\infty$ in $W^{2,q}(\Omega)$ as $p\to\infty$ along a subsequence for any $q>1$. Fix an open subset $\Omega'\subseteq\Omega$ and $\phi\in W^{2,\infty}_0(\Omega')$. Since any open set in $\R$ can be expressed as a countable disjoint union of intervals, we may assume $\Omega'=(a,b)$. In order to conclude, it suffices to show that
\begin{equation}\label{eqn:absoluteminimiserinequalityforab}
  \E_\infty(u_\infty,(a,b))\le\E_\infty(u_\infty+\phi,(a,b))
\end{equation}
for arbitrary $\phi\in W^{2,\infty}_0(a,b)$. Consider, for any $p\in\mathbb{N}\cup\{\infty\}$, the unique cubic polynomial such that $Q_p-u_p\in W^{2,\infty}(a,b)$. By the above observations and Theorem \eqref{theorem:existenceofglobalminimisers}, along a subsequence we have
\begin{equation}\label{eqn:QpconvergesQinfty}
  Q_p\longconv Q_\infty\text{ in }W^{2,\infty}(a,b)\text{ as }p\to\infty.
\end{equation}
We define for $p\in\mathbb{N}$ the function $\phi_p:=\phi+u_\infty-u_p+Q_p-Q_\infty$. Since all three functions $\phi,u_\infty-Q_\infty,u_p-Q_p$ are in $W^{2,\infty}_0(a,b)$, the same is true for $\phi_p$. By Corollary \eqref{corollary:diagonallowersemicontinuity}, and by the additivity of the integral, we have 
\begin{equation}
\begin{aligned}
  \E_p(u_p,(a,b))&\le 2^{-p}+\E_p\big(u_p+\phi_p,(a,b)\big) \\
  &\le2^{-p}+\E_p\big(u_\infty+\phi+[Q_p-Q_\infty],(a,b)\big) \\ 
  &\le 2^{-p}+\bigg(\frac{b-a}{\lvert\Omega\rvert}\bigg)^{\frac1p}\E_\infty\big(u_\infty+\phi+[Q_p-Q_\infty],(a,b)\big).
\end{aligned}
\end{equation}
By invoking \eqref{eqn:QpconvergesQinfty} and passing to the limit as $p\to\infty$, we deduce \eqref{eqn:absoluteminimiserinequalityforab}.
\end{proof}

%%%%%%%%%%%%%%%%%%%%%%%%%%%%%%%%%%%%%%%%%%%%%%%%%%%%%%%%%%%

\smallskip

\section{Variational Characterisation of $\mathrm{A}^2_\infty$ through Absolute Minimisers}\label{section:variationalcharacterisation}

Let us recall that the lack of G\^ateaux differentiability of $\E_\infty$ implies that we cannot simply deduce that minimisers of $\E_\infty$ yield \enquote{stationary points} of the functional through variations. Hence, we follow an alternative approach to derive a PDE as a necessary condition for minimality. To this end, we follow \cite{Katzourakis-Pryer-2020} and use $L^p$ approximations to discover the relevant PDE, in the limit of Euler-Lagrange equations of $L^p$ integral functionals as $p\to\infty$. Of course this well-known formal derivation method does not yield a variational characterisation through absolute minimisers, it allows us merely to discover the ``correct" nonlinear PDE associated to the supremal functional. We then show that absolute minimisers of $\E_\infty$, at least if they are in $C^3(\Omega)$, are classical solutions of the equation \eqref{eqn:A2inftensor} on $\Omega$.

%%%%%%%%%%%%%%%%%%%%%%%%%%%%%%%%%%%%%%%%%%%%%%%%%%%%%%%%%%%%%%

\begin{der}[Formal derivation of $\A^2_\infty u=0$]\label{derivation:derivationofA2infty}
Let $u\in C^4(\Omega)$, and suppose that $\H$ is $C^2$ and bounded below by $-m$, where $m> 0$. Let $p\in (1,\infty)$ and fix $M>m$. Consider the Euler-Lagrange equation of the $p$-integral functional 
$$
\E_p(u,\Omega) :=\bigg(\av_{\Omega} \big(M+\H(\J^2u)\big)^p\bigg)^{\frac1p} -M,
$$ 
which is given by 
\[
\begin{split}
\sum_{i,j=1}^n \D^2_{ij}\left(\big(M+\H(\J^2u)\big)^{p-1}\H_{\X_{ij}}(\J^2u)\right)  & - \sum_{k=1}^n\D_k\left(\big(M+\H(\J^2u)\big)^{p-1}\H_{\p_k}(\J^2u)\right)
\\
& +\big( M +\H(\J^2u)\big)^{p-1} \H_\eta(\J^2u)  =0 \ \ \ \text{ in }\Omega.
\end{split}
\]
By distributing derivatives and rescaling, we obtain 
\[
\begin{aligned}
  \sum_{i,j=1}^n  \D_i\big(\H(\J^2u)\big) \D_j\big(\H(\J^2u)\big)  \H_{\X_{ij}} & (\J^2u) 
  \\
   =\frac{\big(M+\H(\J^2u)\big)^{3-p}}{p-2}\sum_{i,j,k=1}^n\bigg\{ & - \frac{\big(M+\H(\J^2u)\big)^{p-1}}{p-1}\Big[ \H_\eta(\J^2u) 
  \\& - \D_k\big(\H_{\p_k}(\J^2u)\big) - \D^2_{ij}\big(\H_{\X_{ij}}(\J^2u)\big)\Big]\\
  & +  \big(M+\H(\J^2u)\big)^{p-2} \Big[ \D_k\big(\H(\J^2u)\big)\H_{\p_k}(\J^2u)
  \\
  & \hspace{80pt} -  \D_i\big(\D_j\big(\H(\J^2u)\big)\H_{\X_{ij}}(\J^2u)\big)
  \\
  & \hspace{80pt} -  
 \D_i\big(\H(\J^2u)\big)\D_j\big(\H_{\X_{ij}}(\J^2u)\big)\Big] \bigg\}, 
 \end{aligned}
\]
 which implies
 $$
  \sum_{i,j=1}^n  \D_i\big(\H(\J^2u)\big) \D_j\big(\H(\J^2u)\big)  \H_{\X_{ij}}   (\J^2u) =\mathrm O\Big(\frac{1}{p-2} \Big), \ \text{ as } p\to\infty. 
  $$
Since $(M+\H(\J^2u))$ is positive on $\Omega$, we obtain in the limit as $p\to\infty$ that
\[
  \sum_{i,j=1}^n\H_{\X_{ij}}(\J^2u)\D_i\big(\H(\J^2u)\big)\D_j\big(\H(\J^2u)\big)=0 \ \ \text{ in }\Omega,
\]
which in effect is \eqref{eqn:A2inftensor}.  Recall that, since we have the approximation of functionals
\[
\E_p(u,\Omega) \longrightarrow \,\underset{\Omega}{\mathrm{ess}\sup} \, \big(M+\H(\J^2u)\big) - M =\, \E_\infty(u,\Omega),
\]
as $p\to\infty$, this equation is the ``correct" PDE associated to the supremal functional $\E_\infty$.
\end{der}

%%%%%%%%%%%%%%%%%%%%%%%%%%%%%%%%%%%%%%%%%%%%%%%%%%%%%%%%%%%%%%

We now present a tertiary lemma to Theorem \eqref{theorem:variationalcharacterisation}.

%%%%%%%%%%%%%%%%%%%%%%%%%%%%%%%%%%%%%%%%%%%%%%%%%%%%%%%%%%%%%

\begin{lem}\label{lem:danskinstheoremforA2infty}
Let $\Omega\subseteq\R^n$ be open and bounded, and $\H\in C^1(\Omega\times\R\times\R^n\times\R^{n^{\otimes2}}_s)$. Let $u\in C^2(\Omega)$ be an absolute minimiser of $\E_\infty$ on $\Omega$. Fix any $\O\Subset\Omega$, and set
\[
\O(u):=\underset{x\in\overline{\O}}{\mathrm{argmax}}\, \H\big(x,u(x),\D u(x), \D^2 u(x)\big) .
\] 
Then, we have
\begin{equation}
\label{eqn:maxtotalonargmax}
\max_{\O(u)}\bigg[\H_\eta(\J^2u)\phi+\H_\p(\J^2u)\cdot\D\phi+\H_\X(\J^2u):\D^2\phi\bigg]\ge0,
\end{equation} 
and
\begin{equation}
\label{eqn:mintotalonargmax}
\min_{\O(u)}\bigg[\H_\eta(\J^2u)\phi+\H_\p(\J^2u)\cdot\D\phi+\H_\X(\J^2u):\D^2\phi\bigg]\le0,
\end{equation} 
for any $\phi\in C^2_0(\overline{\O}):=W^{2,\infty}_0(\O)\cap C^2(\overline{\O})$.  

Note that \eqref{eqn:mintotalonargmax} follows from \eqref{eqn:maxtotalonargmax} via the substitution $\phi \leftrightarrow -\phi$.
\end{lem}

%%%%%%%%%%%%%%%%%%%%%%%%%%%%%%%%%%%%%%%%%%%%%%%%%%%%%%%%%%%%%

The proof of Lemma \ref{lem:danskinstheoremforA2infty} is essentially an aplication of Danskin's theorem on differentiating maxima for continuous functions \cite{Danskin-1967}. 

%%%%%%%%%%%%%%%%%%%%%%%%%%%%%%%%%%%%%%%%%%%%%%%%%%%%%%%%%%%%%

\begin{proof}[\rm \textbf{Proof of Lemma }\ref{lem:danskinstheoremforA2infty}]
Fix $\O\Subset\Omega$ and $\phi\in C^2_0(\overline{\O})$. Since $u\in C^2(\Omega)$ is an absolute minimiser, we have $\E_\infty(u+\phi,\Omega)\geq \E_\infty(u,\Omega)$. This implies
\begin{equation}
  h(t):=\max_{\overline{\O}} \H\Big(\cdot,u+t\phi,\D u+t\D\phi,\D^2u+t\D^2\phi \Big) \ge\max_{\overline{\O}}\H(\J^2u)=h(0)
\end{equation}
for any $t\in\R$. Hence, if $h'(0^+)$ exists, we must have $h'(0^+)\ge0$, and if $h'(0^-)$ exists, we must have $h'(0^-)\le0$. By Danskin's theorem \cite{Danskin-1967}, we can compute 
\begin{equation}
\begin{aligned}
h'(0^+)&=\frac{\d}{\d t}\bigg|_{t=0^+}\max_{\overline{\O}} \,\H\Big(\cdot,u+t\phi,\D u+t\D\phi,\D^2u+t\D^2\phi\Big) 
\\ 
&=\max_{\overline{\O}}\bigg[\H_\eta(\J^2u)\phi+\H_\p(\J^2u)\cdot\D\phi+\H_\X(\J^2u):\D^2\phi\bigg],
\end{aligned}
\end{equation}
and therefore the semiderivative from the right exists. Similarly, 
\begin{equation}
\begin{aligned}
h'(0^-)&=\frac{\d}{\d t}\bigg|_{t=0^-}\max_{\overline{\O}}\, \H \Big(\cdot,u+t\phi,\D u+t\D\phi,\D^2u+t\D^2\phi \Big)
 \\
  &=\max_{\overline{\O}}\bigg[\H_\eta(\J^2u)\phi+\H_\p(\J^2u)\cdot\D\phi+\H_\X(\J^2u):\D^2\phi\bigg].
\end{aligned}
\end{equation}
The lemma ensues.
\end{proof}

%%%%%%%%%%%%%%%%%%%%%%%%%%%%%%%%%%%%%%%%%%%%%%%%%%%%%%%%%%%%%

\begin{thm}[Absolute minimisers in $C^3$ solve $\A^2_\infty u=0$]\label{theorem:variationalcharacterisation}Let $\Omega\subseteq\R^n$ be open, and suppose that $\H\in C^1(\Omega\times\R\times\R^n\times\R^{n^{\otimes2}}_s)$. If $u\in C^3(\Omega)$ is an absolute minimiser of the functional \eqref{eqn:fullfunctional} on $\Omega$, then $u$ is a classical solution of \eqref{eqn:A2inftensor}, that is 
$$
\A^2_\infty u=\H_\X(\J^2u):\D\big(\H(\J^2u))\otimes\D\big(\H(\J^2u)\big)=0 \ \ \text{ on }\Omega.
$$ 
\end{thm}
%%%%%%%%%%%%%%%%%%%%%%%%%%%%%%%%%%%%%%%%%%%%%%%%%%%%%%%%%%%%%%

The method utilised below is new, and in particular provides a considerably simpler alternative proof of Theorem 14(A) in \cite{Katzourakis-Pryer-2020}, in the case of no lower-order terms.

%%%%%%%%%%%%%%%%%%%%%%%%%%%%%%%%%%%%%%%%%%%%%%%%%%%%%%%%%%%%%%

\begin{proof}[\rm \textbf{Proof of Theorem }\ref{theorem:variationalcharacterisation}]
Fix $x\in\Omega$, $\rho\in(0,\text{dist}(x,\partial\Omega))$. Then, $\overline{\B_\rho(x)}\subseteq\Omega$, i.e. $\B_\rho(x)\Subset\Omega$. Fix also any $\zeta\in C^2[0,1]$ satisfying 
\begin{equation}\label{eqn:conditionsonzeta}
\zeta'(0)=\zeta''(0)=0;\qquad \zeta(1)=\zeta'(1)=0;\qquad\zeta''(1)=1.
\end{equation}
We define $\phi_{x,\rho}\in C^2_0(\overline{\B_\rho(x)})$ by setting 
\begin{equation}
\phi_{x,\rho}(y):=\rho^2\zeta\bigg(\frac{\lvert y-x\rvert}{\rho}\bigg).
\end{equation}
Then, we compute
\[
\left\{\ \ 
\begin{split}
  \D\phi_{x,\rho}(y)=&\ \rho\zeta'\bigg(\frac{\lvert y-x\rvert}{\rho}\bigg)\frac{y-x}{\lvert y-x\rvert},
  \\
  \D^2\phi_{x,\rho}(y)= &\ \zeta''\bigg(\frac{\lvert y-x\rvert}{\rho}\bigg)\frac{y-x}{\lvert y-x\rvert}\otimes\frac{y-x}{\lvert y-x\rvert}
  \\
  &+
  \rho\zeta'\bigg(\frac{\lvert y-x\rvert}{\rho}\bigg)\frac{1}{\lvert y-x\rvert}\bigg[\mathbb{I}-\frac{y-x}{\lvert y-x\rvert}\otimes\frac{y-x}{\lvert y-x\rvert}\bigg].
  \end{split}
  \right.
\]
In view of \eqref{eqn:conditionsonzeta}, $\D\phi_{x,\rho}$ and $\D^2\phi_{x,\rho}$, are continuous on $\overline{\B_\rho(x)}$, and $\phi_{x,\rho}$ as well as $\D\phi_{x,\rho}$ vanish on $\partial\B_\rho(x)$. Further,
\begin{equation}\label{eqn:hessianofphixrhoonboundary}
  \D^2\phi_{x,\rho}(y)=\frac{y-x}{\lvert y-x\rvert}\otimes\frac{y-x}{\lvert y-x\rvert},\ \ \ \text{ for }y\in \partial\B_\rho(x).
\end{equation} 
By Lemma \ref{lem:danskinstheoremforA2infty} for $\O:=\B_\rho(x),\phi:=\phi_{x,\rho}$, we have that there exist $x^\pm\in\overline{\B_\rho(x)}$ (realising the max/min) such that 
\begin{equation}\label{eqn:extremalpoints}
  \begin{cases}
    \bigg(\H_\eta(\J^2u)\phi_{x,\rho}+\H_\p(\J^2u)\cdot\D\phi_{x,\rho}+\H_\X(\J^2u):\D^2\phi_{x,\rho}\bigg)(x_\rho^+)\ge0,\vspace{2mm}\\ 
    \bigg(\H_\eta(\J^2u)\phi_{x,\rho}+\H_\p(\J^2u)\cdot\D\phi_{x,\rho}+\H_\X(\J^2u):\D^2\phi_{x,\rho}\bigg)(x_\rho^-)\le0.
  \end{cases}
\end{equation}
If $x_\rho^+\in\B_\rho(x)$, i.e. if it is an interior maximum of $\H(\J^2u)$ over $\overline{\B_\rho(x)}$, then $\D(\H(\J^2u))(x_\rho^+)=0.$ If $x_\rho^+\in\partial\B_\rho(x)$, then note that 
\begin{equation}\label{eqn:xrhoonboundary}
  \overline{\B_\rho(x)}\subseteq\bigg\{\H(\J^2u)\le\H\left(x_\rho^+,u(x_\rho^+),\D u(x_\rho^+), \D^2u(x_\rho^+)\right)\bigg\}=:\mathscr{H}(x_\rho^+).
\end{equation}
If $\D(\H(\J^2u))(x_\rho^+)=0$, then we have as in the previous case. If $\D(\H(\J^2u))(x_\rho^+)\neq0$, then \eqref{eqn:xrhoonboundary} implies that $\mathscr{H}(x_\rho^+)$ satisfies an interior sphere condition at $x_\rho^+$ and 
\begin{equation}
  \partial\mathscr{H}(x_\rho^+)=\bigg\{\H(\J^2u)=\H\left(x_\rho^+,u(x_\rho^+),\D u(x_\rho^+), \D^2u(x_\rho^+)\right)\bigg\}
\end{equation}
near $x_\rho^+$. Thus the vector $(x-x_\rho)$ is parallel to $ \D(\H(\J^2u))(x_\rho^+)$. Then, \eqref{eqn:extremalpoints} implies (in view of \eqref{eqn:hessianofphixrhoonboundary}) that
\begin{equation}
  \H_\X(\J^2u)(x_\rho^+):\frac{x-x_\rho^+}{\rho}\otimes\frac{x-x_\rho^+}{\rho}\ge0
\end{equation}
which yields, since $(x-x_\rho)$ is parallel to $\D(\H(\J^2u))(x_\rho^+)$, that 
\begin{equation}
  \H_\X(\J^2u)(x_\rho^+):\D\!\left(\H(\J^2u)\right)(x_\rho^+)\otimes\D\!\left(\H(\J^2u)\right)(x_\rho^+)\ge0,
\end{equation}
and the above is also true when $x_\rho^+$ is interior or $\D\!\left(\H(\J^2u)\right)(x_\rho^+)=0$ and $x_\rho^+\in\B_\rho(x)$. Arguing similarly for $x_\rho^-$, we have
\begin{equation}
  \H_\X(\J^2u)(x_\rho^-):\D\!\left(\H(\J^2u)\right)(x_\rho^-)\otimes\D\!\left(\H(\J^2u)\right)(x_\rho^-)\le0.
\end{equation} 
We conclude by letting $\rho\to0^+$.
\end{proof}

%%%%%%%%%%%%%%%%%%%%%%%%%%%%%%%%%%%%%%%%%%%%%%%%%%%%%%%%%%%%%%
\smallskip

\section{Existence of $\mathcal{D}$-solutions for the Dirichlet Problem for $\A^2_\infty u=0$}\label{section:existenceofDsolutions}

In this section we focus on the third-order fully nonlinear PDE \eqref{eqn:A2inftensor}. We study the (first-order) Dirichlet problem for this equation on bounded domains, disregarding its connections to the variational problem for \eqref{eqn:fullfunctional}. We impose rather weak assumptions on $\H$, which may not suffice to guarantee the existence of even global minimisers of \eqref{eqn:fullfunctional}, as it may not be weakly* lower semicontinuous. 

Recall that the operator $\A^2_\infty$ is of third-order, yet the natural regularity class for our variational problem is $W^{2,\infty}(\Omega)$. As shown in \cite{Katzourakis-Pryer-2020}, solutions (in general) to \eqref{eqn:A2inftensor} cannot be classical in $C^3(\Omega)$, not even when $n=1$, there are no lower-order terms, and we restrict our attention to minimising solutions! In particular, for any $a,b,A,B,A',B'\in\R$ with $a<b$, there exists a unique (absolute and global) minimiser of $u \mapsto \|u''\|_{L^\infty(\Omega)}$ satisfying
\[
u(a)=A,u'(a)=A', \ \ u(b)=B,u'(b)=B', 
\]
but unless the endpoint data $A,B,A',B'$ can be interpolated by a quadratic polynomial (in which case this is the absolute minimiser), in all other cases the absolute minimiser is piecewise $C^2$ with exactly one jump discontinuity point in the second derivative. Hence, following \cite{Katzourakis-Pryer-2020}, we also consider generalised $\mathcal{D}$-solutions to \eqref{eqn:A2inftensor}. This concept however applies to the expanded form \eqref{expanded-PDE} of the PDE \eqref{eqn:A2inftensor}, namely
\[
\mathcal{A}_\infty(\J^2u,\D^3u)=0 \ \ \text{ in }\Omega,
\]
which in view of \eqref{Aronsson-operator}, can be written as
\begin{equation}
\label{expanded-PDE-2}
 \H_\X(\J^2u): \! \bigg(\H_x(\J^2u)+\H_\eta(\J^2u)\D u+\H_\p(\J^2u)\D^2u+\H_\X(\J^2u):\D^3u\bigg)^{\!\! \otimes2}=0 \ \ \text{ in }\Omega.
\end{equation}
The two forms of the PDE \eqref{expanded-PDE-2} and \eqref{eqn:A2inftensor} are of course equivalent when $u\in C^3(\Omega)$, but the expanded form is the appropriate one to define $\mathcal{D}$-solutions, and our method of proof is inspired by the equivalence between the two. The idea is if for a function $u\in C^3(\Omega)$ we have $\H(\J^2u)\equiv C$ on $\Omega$ for some $C\in\R$, then $\A^2_\infty u =0$ in $\Omega$. The main technical obstacle is then to show that if we have merely that $u\in W_g^{2,\infty}(\Omega)$, which is the natural regularity class, then $\mathcal{A}_\infty(\J^2u,\D^3u)=0$ in the $\mathcal{D}$-sense in $\Omega$.

This section is structured as follows. We first prove the existence of infinitely many strong a.e.\ solutions in $W_g^{2,\infty}(\Omega)$ to the Dirichlet problem for the second-order implicit PDE
\begin{equation}
  \H(\J^2u(x))=C \ \text{ for a.e. }x\in\Omega,
\end{equation}
for some compatible $C>0$ large enough, understood as an \enquote{energy level} determined on the boundary data $g\in W^{2,\infty}(\Omega)$. For this we utilise the results of \cite{Dacorogna-Marcellini-1999} on implicit fully nonlinear PDE, which rely on the Baire Category method. We then establish that all strong solutions $u\in W^{2,\infty}_g(\Omega)$ to $\H(\J^2u)=C$ a.e. on $\Omega$ are $\mathcal{D}$-solutions to the Dirichlet problem for $\mathcal{A}_\infty(\J^2u,\D^3u)=0$ on $\Omega$. To this end, we utilise some basic general machinery of $\mathcal{D}$-solutions established in \cite{Katzourakis-2017-Dsols}. As a result, we provide a proof which is not only new even in the special setting of \cite{Katzourakis-Pryer-2020} in which no lower-order terms are considered, but also considerably shorter than the ``first principles" proof provided therein.

We begin with the existence of solutions to our implicit second-order PDE.

%%%%%%%%%%%%%%%%%%%%%%%%%%%%%%%%%%%%%%%%%%%%%%%%%%%%%%%%%%%%%%

\begin{lem}[Solvability of $\H(\J^2u)=C$]\label{lem:existenceofsolutionstoH=C}
Suppose that $\H:\Omega\times\R\times\R^n\times\R^{n^{\otimes2}}_s \to\R$ satisfies
$$
\H(x,\eta,\mathrm p, \X)=h(x,\eta,\mathrm p, \X^\top\! \X),
$$ 
for some $h\in C^1\big(\Omega\times \R\times\R^n\times\R^{n^{\otimes2}}_s\big)$. We further assume that, for any ${(x,\eta,\mathrm p)\in\Omega\times\R\times\R^n}$, $h(x,\eta,\mathrm p,\cdot)$ is strictly increasing along the direction of the identity matrix, namely the function $t\mapsto h(x,\eta,\mathrm p,t\mathbb{I})$ is strictly increasing on $\R$. Further, we assume there exists $\delta_0>0$ such that 
\begin{equation}
\label{5.3}
\underset {\Omega \times \R \times \R^n}{\sup} h \big( \cdot,\cdot,\cdot, \delta_0 \mathbb{I}\big) <\infty.
\end{equation}
Then, for any $C>0$ selected such that
\begin{equation}
\label{5.4}
C \geq \max\left\{ \underset {\Omega \times \R \times \R^n}{\sup} h \big( \cdot,\cdot,\cdot, \delta_0 \mathbb{I}\big)\, ,\, \underset {\Omega }{\mathrm{ess}\sup}\, h\Big(\J^1 g, \Big( 1  +  \|\D^2g\|^2_{L^\infty(\Omega)}\Big)\mathbb{I}\Big)  \right\},
\end{equation}
the Dirichlet problem
\begin{equation}
\label{5.5A}
  \begin{cases}
\  \H(\J^2u)=C, \ & \text{a.e. in }\Omega, \\
 \ u=g, \D u=\D g, \ & \text{on }\partial\Omega,
  \end{cases}
\end{equation}
has (infinitely many) strong solutions $u\in W^{2,\infty}_g(\Omega)$.
\end{lem}

%%%%%%%%%%%%%%%%%%%%%%%%%%%%%%%%%%%%%%%%%%%%%%%%%%%%%%%%%%%%%%

\begin{proof}[\rm \textbf{Proof of Lemma }\ref{lem:existenceofsolutionstoH=C}]
Fix $g\in W^{2,\infty}(\Omega)$. For any matrix $\X\in \R^{n^{\otimes 2}}_s$, let $\{\lambda_1(\X)$, ...,$\lambda_n(\X)\}$ denote its eigenvalues in increasing order. Let $f : \Omega \times \R \times \R^n \to \R$ be a continuous function satisfying $f\ge\delta_0$ for some $\delta_0>0$. By \cite[Theorem 7.31, Remark 7.29 \& Corollary 7.34]{Dacorogna-Marcellini-1999}, the following Dirichlet problem involving singular values of the Hessian matrix
\begin{equation}
\label{5.5}
  \ \ \ \begin{cases}
\ \lambda_i(\D^2u^\top \D^2u)=f(\J^1 u), & \text{a.e. on }\Omega, \quad i=1,\ldots,n, \\
  \ u=g, \, \D u=\D g, & \text{on }\partial\Omega,
  \end{cases}
\end{equation}
has (infinitely many) solutions in $W^{2,\infty}_g(\Omega)$, as long as the boundary condition is a ``strict subsolution", in the sense that
\begin{equation}
\label{5.6}
 \lambda_n(\D^2g^\top \D^2g)\le f(\J^1g)-\varepsilon_0 \ \ \text{ a.e. on }\Omega,
\end{equation}
for some $\varepsilon_0>0$. When such solutions exist, by the spectral theorem we have that there exists a measurable map of orthogonal matrices $\O \in L^\infty\big(\Omega;\mathrm O(n,\R)\big)$ such that
\begin{equation}
\label{5.8}
  \begin{aligned}
  \D^2u^\top \D^2u&=\O\begin{bmatrix}\lambda_1(\D^2u^\top \D^2u)& &\mathbb{O}\\ & \ddots & \\\mathbb{O}& &\lambda_n(\D^2u^\top \D^2u)\end{bmatrix}\O^\top  \\ 
  &=f(\J^1u)\O\mathbb{I}\O^\top  \\ 
  &=f(\J^1u)\mathbb{I},
  \end{aligned}
\end{equation}
a.e.\ on $\Omega$. Let $C>0$ be selected as in \eqref{5.4}. It follows that
\[
C \geq h\Big(\J^1 g, \big( 1+\|\D^2g\|^2_{L^\infty(\Omega)}\big)\mathbb{I}\Big) \ \ \text{a.e.\ on }\Omega,
\]
and by the monotonicity assumption on $h$, this implies
\[
h\big(\J^1 g,(\cdot)\mathbb{I}\big)^{-1} (C) \geq  1+\|\D^2g\|^2_{L^\infty(\Omega)}  \ \ \text{a.e.\ on }\Omega.
\]
Therefore, by selecting 
\begin{equation}
\label{5.9}
  f(x,\eta,\mathrm p) := h(x,\eta,\mathrm p,(\cdot)\mathbb{I})^{-1}(C),
\end{equation}
we obtain
\[
  f(\J^1g)\ge1+\lVert\D^2g\rVert_{L^\infty(\Omega)}^2
\]
pointwise on $\Omega$. Again by \eqref{5.3}-\eqref{5.4}, it follows that there exists $\delta_0>0$ such that
\[
C \geq h \big( x,\eta, \mathrm p, \delta_0 \mathbb{I}\big).
\]
for all $(x,\eta, \mathrm p) \in \Omega \times \R \times \R^n$. This yields that
\[
f(x,\eta,\mathrm p) = h(x,\eta,\mathrm p,(\cdot)\mathbb{I})^{-1}(C) \geq \delta_0, \ \ \text{ a.e.\ on }\Omega. 
\]
It follows that the necessary conditions for the solvability of problem \eqref{5.5} are satisfied, as 
\[
  \begin{aligned}
\lambda_n(\D^2g^\top \D^2g)&\le\lVert\D^2g\rVert_{L^\infty(\Omega)}^2 \\
    &=(1+\lVert\D^2g\rVert_{L^\infty(\Omega)}^2)-1 \\
    &\le f(\J^1g)-1,
  \end{aligned}
\]
a.e. on $\Omega$, which is \eqref{5.6} with $\varepsilon_0=1$. We conclude by showing how the solvability of \eqref{5.5} implies the solvability of \eqref{5.5A}. In view of \eqref{5.8}-\eqref{5.9}, we may compute
\[
\begin{split}
\H(\J^2 u ) & = h\big(\J^1u,\D^2u^\top \D^2 u \big)
\\
& = h\big(\J^1u,f(\J^1u)\mathbb{I} \big)
\\
& = h\big(\J^1u,(\cdot)\mathbb{I} \big) \big(f(\J^1u)\big)
\\
& = \Big(h\big(\J^1u,(\cdot)\mathbb{I} \big) \circ h\big(\J^1u,(\cdot)\mathbb{I} \big)^{-1}\Big)(C)
\\
& = C,
\end{split}
\]
a.e.\ on $\Omega$. The conclusion ensues.
\end{proof}

Now we proceed to establish the existence of $\mathcal{D}$-solutions to the Dirichlet problem for the fully nonlinear PDE \eqref{expanded-PDE-2}.

%%%%%%%%%%%%%%%%%%%%%%%%%%%%%%%%%%%%%%%%%%%%%%%%%%%%%%%%%%%%%%

\begin{thm}[Existence of $\mathcal{D}$-solutions]\label{theorem:existenceofDsolutions}
Suppose that $\H:\Omega\times\R\times\R^n\times\R^{n^{\otimes2}}_s \to\R$ satisfies
$$
\H(x,\eta,\mathrm p, \X)=h(x,\eta,\mathrm p,\X^\top\! \X),
$$ 
for some $h\in C^1\big(\Omega\times \R\times\R^n\times\R^{n^{\otimes2}}_s\big)$. We further assume that, for any ${(x,\eta,\mathrm p)\in\Omega\times\R\times\R^n}$, $h(x,\eta,\mathrm p,\cdot)$ is strictly increasing along the direction of the identity matrix, and that there exists $\delta_0>0$ such that 
\[
\underset {\Omega \times \R \times \R^n}{\sup} h \big( \cdot,\cdot,\cdot, \delta_0 \mathbb{I}\big) <\infty.
\]
Then, for any $g\in W^{2,\infty}(\Omega)$, the Dirichlet problem
\begin{equation}
\label{5.10}
  \ \ \ \begin{cases}
\ \mathcal{A}_\infty(\J^2u,\D^3u) =0, & \text{in }\Omega,
 \\
  \ u=g, \, \D u=\D g, & \text{on }\partial\Omega,
  \end{cases}
\end{equation}
has (infinitely many) $\mathcal{D}$-solutions $u\in W^{2,\infty}_g(\Omega)$. In view of Definition \ref{D-solutions} and the expression \eqref{Aronsson-operator} of $\mathcal{A}_\infty$, this means that 
\[
\int_{\overbar{\R}^{n^{\otimes3}}_s} \! \Phi(\bZ) \H_\X(\J^2u) \!: \!\!\bigg(\H_x(\J^2u)+\H_\eta(\J^2u)\D u+\H_\p(\J^2u)\D^2u+\H_\X(\J^2u)\!:\!\bZ\bigg)^{\!\! \otimes2}\d[\mathcal{D}^3u](\bZ) =0 ,
\] 
a.e.\ on $\Omega$, for any third-order diffuse derivative $\mathcal{D}^3u\in\mathscr{Y}(\Omega;\overbar{\R}^{n^{\otimes3}}_s)$, and for any test function $\Phi\in C_c(\overbar{\R}^{n^{\otimes3}}_s)$.
\end{thm}

%%%%%%%%%%%%%%%%%%%%%%%%%%%%%%%%%%%%%%%%%%%%%%%%%%%%%%%%%%%%%%

\begin{proof}[\rm \textbf{Proof of Theorem }\ref{theorem:existenceofDsolutions}]
Since the assumptions of Lemma \ref{lem:existenceofsolutionstoH=C} are satisfied, we may select $C>0$ as in \eqref{5.4}, to guarantee that the Dirichlet problem \eqref{5.5A} for the fully nonlinear PDE $\H(\J^2u)=C$ is solvable in the strong sense in $W^{2,\infty}_g(\Omega)$. By the result \cite[Theorem 30, p. 665]{Katzourakis-2017-Dsols} on the differentiation of equations in the $\mathcal{D}$-sense, $u$ solves 
\begin{equation}
  \sum_{i,j=1}^n\bigg(\H_{x_k}(\J^2u)+\H_\eta(\J^2u)\D_k u+\H_{\p_j}(\J^2u)\D^2_{jk}u+\H_{\X_{ij}}(\J^2u)\D^3_{kij}u\bigg)=0 \ \ \text{ in }\Omega,
\end{equation}
in the $\mathcal{D}$-sense, for any $k=1,\ldots,n$. We define the map
$ \mathscr{L}_\infty: \Omega\times\R\times\R^n\times\R^{n^{\otimes2}}_s \times\R^{n^{\otimes3}}_s\to\R^n$ by setting 
\begin{equation}
\label{5.12}
  \begin{aligned}
  \mathscr{H}_\infty(x,\eta,\mathrm p , \X ,\bZ):=\sum_{i,j,k=1}^n\bigg(&\H_{x_k} (x,\eta,\mathrm p , \X )+\H_\eta(x,\eta,\mathrm p , \X )\mathrm \mathrm p_k\\&+\H_{\p_j}(x,\eta,\mathrm p , \X )\X_{jk}+\H_{\X_{ij}}(x,\eta,\mathrm p , \X )\bZ_{kij}\bigg)e^k.
  \end{aligned}
\end{equation}
Then, in view of Remark \ref{remark2.4}, we have that any third-order diffuse derivative $\mathcal{D}^3u \in\mathscr{Y}(\Omega;\overbar{\R}^{n^{\otimes3}}_s)$ of $u$ satisfies the inclusion 
\[
  \mathrm{supp}_*(\mathcal{D}^3u(x))\subseteq\big\{\mathscr{H}_\infty(\J^2u(x),\cdot)=0\big\},
\]
for a.e.\ $x\in\Omega$. Note now that the definitions \eqref{Aronsson-operator} and \eqref{5.12} imply that
\[
\mathcal{A}_\infty(\J^2u ,\bZ) = \H_\X(\J^2u) : \mathscr{H}_\infty(\J^2u ,\bZ) \otimes \mathscr{H}_\infty(\J^2u ,\bZ),
\]
a.e.\ on $\Omega$. Therefore, we have 
\[
\big\{\mathscr{H}_\infty(\J^2u(x),\cdot)=0\big\}\subseteq\big\{\mathcal{A}_\infty(\J^2u(x),\cdot)=0\big\} 
\]
for a.e.\ $x\in\Omega$, which by the above yields that 
\[
  \mathrm{supp}_*(\mathcal{D}^3u(x))\subseteq\big\{\mathcal{A}_\infty(\J^2u(x),\cdot)=0\big\},
\]
for a.e.\ $x\in\Omega$. Therefore, by Remark \ref{remark2.4} again, it follows that $u$ is a $\mathcal{D}$-solution to $\mathcal{A}_\infty(\J^2u,\D^3u)=0$ in $\Omega$. The conclusion ensues.
\end{proof}

%%%%%%%%%%%%%%%%%%%%%%%%%%%%%%%%%%%%%%%%%%%%%%%%%%%%%%%%%%%%%%

%\vskip.5in

\bibliographystyle{amsplain}

\begin{thebibliography}{30}

   \bibitem{Katzourakis-Abugirda-2016}
  H. Abugirda, N. Katzourakis, \emph{Existence of $1D$ Vectorial Absolute Minimisers in $L^\infty$ under Minimal Assumptions}, Proceedings of the American Mathematical Society, \textbf{145}, 2567--2575 (2016).
 

  \bibitem{AP} N. Ansini, F. Prinari, \emph{On the lower semicontinuity of supremal functional under differential constraints}, ESAIM - Control, Opt. and Calc. Var., \textbf{21} (4), 1053--1075 (2015).

  \bibitem{Katzourakis-Ayanbayev-2017}
  B. Ayanbayev, N. Katzourakis, \emph{A Pointwise Characterisation of the PDE System of Vectorial Calculus of Variations in $L^\infty$}, Proceedings of the Royal Society of Edinburgh Section A Mathematics, \textbf{150} (4), 1653--1669 (2017).

  \bibitem{Aronsson-1965}
  G. Aronsson, \emph{Minimization problems for the functional $sup_x F(x, f(x), f'(x))$}, Arkiv f\"ur Matematik, \textbf{6}, 33--53 (1965).
  
  \bibitem{Aronsson-1966}
  G. Aronsson, \emph{Minimization problems for the functional $sup_x F(x, f(x), f'(x))$. (II)}, Arkiv f\"ur Matematik, \textbf{6}, 409--431 (1966).
  
  \bibitem{Aronsson-1967}
  G. Aronsson, \emph{Extension of functions satisfying Lipschitz conditions}, Arkiv f\"ur Matematik, \textbf{6}, 551--561 (1967).
  
  \bibitem{Aronsson-1984}
  G. Aronsson, \emph{On certain singular solutions of the partial differential equation $u_x^2u_{xx} + 2u_xu_yu_{xy} + u_y^2u_{yy} = 0$}, Manuscripta Mathematica, \textbf{47}, 133--151 (1984).
  
  \bibitem{Aronsson-1986}
  G. Aronsson, \emph{Construction of singular solutions to the $p$-harmonic equation and its limit equation for $p=\infty$}, Manuscripta Mathematica, \textbf{56}, 135--158 (1986).
  
  \bibitem{Aronsson-Barron} G. Aronsson, E. N. Barron, \emph{$L^\infty$ variational problems with running costs and constraints}, Appl. Math. Optim. 65 no. 1, 53--90 (2012).


  \bibitem{BJW-2001-Euler}
  E. N. Barron, R. Jensen, and C. Wang, \emph{The Euler Equation and Absolute Minimizers of $L^\infty$ Functionals}, Archive for Rational Mechanics and Analysis, \textbf{157}, 255--283 (2001).
  
  \bibitem{BJW-2001-Lsc}
  E. N. Barron, R. Jensen, and C. Wang, \emph{Lower semicontinuity of $L^\infty$ functionals}, Annales de l'Institut Henri Poincar\'e C, Analyse non linéaire, \textbf{18} (4), 495--517 (2001).

  \bibitem{Bhattacharya-DiBenedetto-Manfredi-1989}
  T. Bhattacharya, E. DiBenedetto, and J. Manfredi, \emph{Limits as $p\to\infty$ of $\nabla_pu_p = f$ and related extremal problems}, Rend. Sem. Mat. Univ. Politec. Torino, \textbf{47}, 15--68 (1989).
  
  \bibitem{Clark-Katzourakis-2024}
  E. Clark and N. Katzourakis, \emph{Generalized second order vectorial $\infty$-eigenvalue problems}, Proc. R. Soc. Edinb. A Math. \textbf{154}, 1--21 (2024).
  
    \bibitem{Croce-Katzourakis-Pisante-2017}
  G. Croce, N. Katzourakis, and G. Pisante, \emph{$\mathcal{D}$-solutions to the system of vectorial Calculus of Variations in $L^\infty$ via the singular value problem}, Discrete and Continuous Dynamical Systems, \textbf{37} (12), 6165--6181 (2017).
    
  
  \bibitem{Dacorogna-Marcellini-1999}
  B. Dacorogna and P. Marcellini, \emph{Implicit partial differential equations}, Progress in Nonlinear Differential Equations and Applications, Birkh\"auser (1999)    
  
  \bibitem{Danskin-1967}
  J. M. Danskin, \emph{The Theory of Max-Min and its Application to Weapons Allocation Problems}, Heidelberg Springer Berlin (1967).
    
  \bibitem{Fonseca-Leoni-2007}
  I. Fonseca and G. Leoni, \emph{Modern Methods in the Calculus of Variations: $L^p$ Spaces}, Springer Monographs in Mathematics, Springer New York (2007)

  \bibitem{Florescu-Godet-Thobie-2012}
  L. C. Florescu and C. Godet-Thobie, \emph{Young Measures and Compactness in Measure Spaces}, De Gruyter (2012)
  

  \bibitem{Katzourakis-2015-AbsMin}
  N. Katzourakis, \emph{Absolutely Minimising Generalised Solutions to the Equations of Vectorial Calculus of Variations in $L^\infty$}, Calculus of Variations, \textbf{56} (15) (2015).
 
  
    \bibitem{Katzourakis-2017-Dsols}
  N. Katzourakis, \emph{Generalised Solutions for Fully Nonlinear PDE Systems and Existence-Uniqueness Theorems}, Journal of Differential Equations, \textbf{263} (1), 641--686 (2017).

    \bibitem{Katzourakis-Moser-2019}
  N. Katzourakis and R. Moser, \emph{Existence, Uniqueness and Structure of Second Order Absolute Minimisers}, Archive for Rational Mechanics and Analysis, \textbf{231}, 1615--1634 (2019).
  
    \bibitem{Katzourakis-Moser-2023}
  N. Katzourakis and R. Moser, \emph{Variational problems in $L^\infty$ involving semilinear second order differential operators}, ESAIM Control Optim. Calc. Var. \textbf{29}, 1--30 (2023).

  \bibitem{Katzourakis-Moser-2024-1Currents}
  N. Katzourakis and R. Moser, \emph{Minimisers of supremal functionals and mass-minimising 1-currents}, Calc. Var. Partial Differential Equations, in press (2024).
  
   \bibitem{Katzourakis-Moser-2024-LocalMinimisers}
  N. Katzourakis and R. Moser, \emph{Existence, uniqueness and characterisation of local minimisers in higher order calculus of variations in $L^\infty$}, preprint \url{https://arxiv.org/abs/2403.12625} (2024).

  \bibitem{Katzourakis-Parini-2017}
  N. Katzourakis and E. Parini, \emph{The eigenvalue problem for the $\infty$-Bilaplacian}, Nonlinear Differential Equations Appl. NoDEA \textbf{24}, 1--25 (2017).
  
  \bibitem{Katzourakis-Pryer-2020}
  N. Katzourakis and T. Pryer, \emph{Second Order $L^\infty$ Variational Problems and the $\infty$-Polylaplacian}, Advances in Calculus of Variations, \textbf{13} (2), 115--140 (2020).

  \bibitem{Katzourakis-Shaw-2018}
  N. Katzourakis and G. Shaw, \emph{Counterexamples in calculus of variations in $L^\infty$ through the vectorial Eikonal equation}, C. R. Math. \textbf{356} (5), 498--502 (2018).
  

       \bibitem{KZ} C. Kreisbeck, E. Zappale, \emph{Lower semicontinuity and relaxation of nonlocal $\mathrm L^{\infty}$-functionals}, Calculus of Variations and PDE, \textbf{59} (138), 1--36 (2020).


  \bibitem{MWZ} Q. Miao, C. Wang, Y. Zhou, \emph{Uniqueness of Absolute Minimizers for $\mathrm L^{\infty}$-Functionals Involving Hamiltonians $H(x,p)$}, Archive for Rational Mechanics and Analysis \textbf{223} (1), 141--198 (2017).
  
  \bibitem{PP} G. Papamikos, T. Pryer, \emph{A Lie symmetry analysis and explicit solutions of the two-dimensional $\infty$-Polylaplacian}, Studies in applied mathematics, \textbf{142} (1), 48--64 (2019).
  
  \bibitem{PZ} F. Prinari, E. Zappale, \emph{A Relaxation Result in the Vectorial Setting and Power Law Approximation for Supremal Functionals}, J Optim. Theory Appl., \textbf{186}, 412--452 (2020).

  \bibitem{RZ} A.N. Ribeiro, E. Zappale, \emph{Existence of minimisers for nonlevel convex functionals}, SIAM J. Control Opt., \textbf{52} (5), 3341--3370 (2014).

  \bibitem{Ribeiro-Zappale-2024}
  A. M. Ribeiro and E. Zappale, \emph{Revisited convexity notions for $L^\infty$ variational problems}, Revista Matem\'atica Complutense, published online, https://doi.org/10.1007/s13163-024-00499-0 (2024)
  

  
  
  

  
  
  
  
  

  

  
 
  
 
  

  
 

\end{thebibliography}

\end{document}